\documentclass[leqno]{siamltex}

\usepackage{amsmath,  amssymb, enumerate, xcolor, color, bm}
\usepackage{ graphicx, float, wasysym, mathrsfs}
\usepackage{stmaryrd,xfrac, bbm}
\usepackage[utf8]{inputenc}
\usepackage{geometry}
\usepackage{marginnote}
\usepackage{enumerate}
\usepackage{enumitem}
\newtheorem{remark}{Remark}

\newtheorem{thm}{Theorem}[section]
\newtheorem{cor}[thm]{Corollary}

\newcommand\mz[1]{{\color{blue}{#1}}}
\newcommand\mzdet[1]{{\color{purple}{#1}}}	
\newcommand\mrg[1]{{\marginnote{\footnotesize{\color{blue}{#1}}}}}
\newcommand\mgg[1]{{\marginnote{\footnotesize{\color{red}{#1}}}}}
\newif\ifdetail
\newcommand{\LL}{l}	
\newcommand{\non}{N_\LL}	
\newcommand{\Cr}{C_{r}}	
\newcommand{\CR}{C_{R}}	
\newcommand{\Clc}{C_{\text{lc}}}	
\newcommand{\Csm}{C_{\text{sm}}}	
\newcommand{\CLC}{C_{\text{LC}}}	
\newcommand{\CSM}{C_{\text{SM}}}	
\newcommand{\Cen}{C_{\text{en}}}	
\newcommand{\CA}{C_{A}}			
\newcommand{\CF}{C_{\text{F}}}		
\newcommand{\Cinv}{C_{\text{inv}}}	
\newcommand{\Ccon}{C_{\text{con}}}	
\newcommand{\Cacon}{\tilde{C}_{\text{con}}}	
\newcommand{\Cerr}{C_{\text{err}}}				
\newcommand{\Cerrt}{\tilde{C}_{\text{err}}}				
\newcommand{\Cco}{C_{\text{co}}}				
\newcommand{\pen}{\mu_\face}					
\newcommand{\CT}{C_{\mathrm{T}}}				

\newcommand{\smth}{s}								
\newcommand{\prmcnst}{\mathcal{E}_{p}}				
\newcommand{\dualcnst}{\mathcal{E}_{d}}				
\newcommand{\er}{e}			
\newcommand{\prer}{\eta}		
\newcommand{\dser}{\xi}		

\newcommand{\tr}{^t}		

\newcommand{\inprods}[2]{{#1} \cdot {#2}}

\newcommand{\dom}{\Omega}
\newcommand{\telem}{\kappa}			
\newcommand{\face}{e}

\newcommand{\intk}{\int_\telem}	
\newcommand{\intdk}{\int_{\partial\telem}}	
\newcommand{\no}{\nu}

\newcommand{\h}{h}
\newcommand{\hk}{\h_\telem}
\newcommand{\he}{\h_\face}

\newcommand{\triag}{\mathcal T_\h}

\newcommand{\inte}{\int_\Gamma}
\newcommand{\intdom}{\int_\dom}
\newcommand{\inted}{\int_\face}
\newcommand{\df}[1]{\,\mathrm{d}#1}
\newcommand{\dx}{\df{x}}
\newcommand{\ds}{\df{s}}

\newcommand{\sumk}{\sum_{\telem \in \triag} }
\newcommand{\sume}{\sum_{\face \in \sklt} }		
\newcommand{\sumei}{\sum_{\face \in \sklti} }		
\newcommand{\sklt}{\mathcal{E}_\h }		
\newcommand{\sklti}{\mathcal{E}_{\h, I} }		
\newcommand{\skltb}{\mathcal{E}_{\h, \partial} }		

\newcommand{\liftnorm}[1]{\seminorm{{#1}}_{{*}, \h} }	
\newcommand{\normgen}[1]{|\!|#1|\!|}	
\newcommand{\seminorm}[1]{|#1|}		
\newcommand{\ennorm}[1]{ |\!\normgen{#1}\!|_{\h} }		

\newcommand{\projh}{\pi_\h}			
\newcommand{\galproj}{\mathcal{G}_h}		

\newcommand{\poldeg}{q}
\newcommand{\altpoldeg}{p}
\newcommand{\polsp}{\mathcal P}
\newcommand{\polp}{\polsp^\poldeg}
\newcommand{\polaltp}{\polsp^\altpoldeg}

\newcommand{\Vh}{V(\h)}					
\newcommand{\Vhd}{V'(\h)}				

\newcommand{\Vsp}{V_{\h,\poldeg}}

\newcommand{\Sigsp}{\Sigma_{\h,\altpoldeg}}

\newcommand{\Lp}[1]{{L_{#1}}}
\newcommand{\Ltwo}{\Lp{2}}
\newcommand{\Linf}{\Lp{\infty}}
\newcommand{\Wmp}[2]{W_{#1}^{#2}} 
\newcommand{\Hm}[1]{H^#1} 

\newcommand{\ltwonrm}[2]{\Vert #1 \Vert_{\Ltwo(#2)}}

\newcommand{\uu}{u}					
\newcommand{\uh}{\uu_\h}			

\newcommand{\tet}{\bm{\theta}}	
\newcommand{\teth}{\tet_h}			

\newcommand{\sig}{\bm{\sigma}}	
\newcommand{\sigh}{\sig_h}			

\newcommand{\vv}{v}			
\newcommand{\vh}{\vv_h}			

\newcommand{\ww}{w}
\newcommand{\wh}{\ww_h}
\newcommand{\xih}{\xi_h}

\newcommand{\zz}{z}

\newcommand{\zeth}{\bm{\zeta}_h}			
\newcommand{\tauh}{\bm{\tau}_h}			
\newcommand{\taub}{\bm{\tau}}			

\newcommand{\liftR}{r}
\newcommand{\liftRe}{\liftR^\face}

\newcommand{\liftL}{l}
\newcommand{\liftLe}{\liftL^\face}

\newcommand{\diffSym}{a}
\newcommand{\diff}{\bm{\diffSym}}
\newcommand{\sdiff}{\diffSym}
\newcommand{\diffz}{\diff_{\bm{z}}}
\newcommand{\diffu}{\diff_{u}}
\newcommand{\difftu}{\tilde{\diff}_u}
\newcommand{\difft}{\tilde{\diff}_{\bm{z}}}

\newcommand{\divg}{\nabla \cdot}
\newcommand{\divh}{\nabla_h \cdot}
\newcommand{\gradh}{\nabla_h }
\newcommand{\uflux}{\hat{u}}
\newcommand{\sflux}{\hat{\sig}}
\newcommand{\jump}[1]{\llbracket #1 \rrbracket}
\newcommand{\avg}[1]{\{ \! \! \{ #1 \}\! \! \} }

\newcommand{\rhsfrm}{\mathcal{F}}			
\newcommand{\tarfrm}{\mathcal{J}}

\newcommand{\difffrm}{\mathcal{B}}

\newcommand{\sipg}{{\mathrm{SIPG}}}			
\newcommand{\brtwo}{{\mathrm{BR2}}}			
\newcommand{\brone}{{\mathrm{BR1}}}			

\newcommand{\Tone}{T^{(1)}}
\newcommand{\Ttwo}{T^{(2)}}

\newcommand{\bldg}{\bm{\beta}}
\newcommand{\Real}{\mathbb{R}}

\title{Analysis of mixed discontinuous Galerkin formulations for quasilinear elliptic problems}

\author{Mohammad Zakerzadeh\footnotemark[2] \and Georg May\footnotemark[2]  \thanks{The research of the authors was supported by the
Deutsche Forschungsgemeinschaft (German Research Association) through grant GSC 111.} }
         
\detailfalse 

\begin{document}

\ifdetail
\mrg{
\mzdet{Here we just consider the quasilinear case, based on e.g. "Linear and quasilinear elliptic equations" by Ladyzhenskaya}
}
\fi 
\maketitle

\renewcommand{\thefootnote}{\fnsymbol{footnote}}
\footnotetext[1]{Aachen Institute for Advanced Study in Computational Engineering Science, RWTH Aachen, 52062 Aachen, Germany 
        (\texttt{\{zakerzadeh, may\}@aices.rwth-aachen.de}).}
\renewcommand{\thefootnote}{\arabic{footnote}}

\begin{abstract}
In this manuscript we present an approach to analyze the discontinuous Galerkin solution for general quasilinear elliptic  problems. This approach is sufficiently general to extend most of the well-known discretization schemes, including {BR1, BR2, SIPG and LDG},  to nonlinear cases in a canonical way, and to establish the stability of their solution. Furthermore, in case of monotone and globally Lipschitz problems, we prove the existence and uniqueness of the approximated solution and the $\h$-optimality of the error estimate in the energy norm as well as in the $\Ltwo$ norm.
\end{abstract}

\begin{keywords} 
discontinuous Galerkin methods, nonlinear elliptic problems, optimal estimates
\end{keywords}

\begin{AMS}
65N12, 65N15, 65N30
\end{AMS}

\pagestyle{myheadings}
\thispagestyle{plain}
\markboth{\uppercase{M. Zakerzadeh AND G. May}}{\uppercase{Analysis of  DG for quasilinear elliptic problems}}

\section{Introduction}
\label{Sec::Int} 

A great amount of research has been devoted to the analysis of discontinuous Galerkin (DG) schemes for (non)linear elliptic and parabolic problems. This goes back to work of Babuska \cite{babuvska1973nonconforming}, Wheeler \cite{wheeler1978elliptic}, Arnold \cite{arnold1982interior} and Dupont et al.\ \cite{10.2307/2005280} for interior penalty (IP) methods. In later  work, e.g. \cite{baumann1999discontinuous, riviere2001priori}, a non-symmetric IP scheme was presented and analyzed. Other classes of discretization techniques have also been introduced, e.g., the first and second method of Bassi--Rebay \cite{bassi1997high1, bassi1997high}, 
or Shu and Cockburn's LDG method \cite{cockburn1998local}. For a comprehensive literature study of these problems we refer to \cite{cockburn2000development} and \cite{arnold2002unified}.

{In their seminal paper \cite{arnold2002unified}, Arnold et al.  provided a unified framework for obtaining different classes of DG methods for the linear Poisson problem. For nonlinear problems the picture is more complicated, and obtaining different classes of methods usually follows very different approaches; while IP methods often have been introduced and analyzed in the primal form  \cite{dolejvsi2008analysis, gudi2008hp, gudi2007discontinuous, ortner2007discontinuous, houston2005discontinuous}, for LDG methods the usual approach is to start from a mixed formulation, e.g., \cite{bustinza2004local, gudi2008hpmixed, bustinza2005mixed} or \cite{yadav2013superconvergent}. 
Originating from this, different techniques are usually employed in the analysis of each family. The LDG methods \cite{gudi2008hpmixed, bustinza2004local, bustinza2005mixed, yadav2013superconvergent} have been formulated by using three unknowns in the mixed form as  proposed in \cite{cockburn1999some}, which yields an inconsistent primal formulation as analyzed, e.g., in \cite{bustinza2004local} (and \cite{perugia2002hp} for the simpler linear case). In contrast, IP methods like \cite{gudi2007discontinuous, gudi2008hp, houston2005discontinuous, ortner2007discontinuous, dolejvsi2008analysis} usually enjoy a consistent primal form.

On the other hand, the nonlinear versions of the Bassi-Rebay methods are often obtained by ad-hoc extension, mimicking the linear counter part, e.g. \cite{bassi2005discontinuous}, and to the best knowledge of the authors,  there does not exist rigorous analysis of the second Bassi--Rebay method for nonlinear problems,	 although they are widely used for nonlinear problems like the Navier--Stokes equations.

This manuscript aims to extend the discussion presented in \cite{bustinza2004local, gudi2007discontinuous, gudi2008hpmixed, gudi2008hp, houston2005discontinuous, ortner2007discontinuous}  in two ways: Firstly, we try to treat different types of discretization in a canonical way, starting from a mixed formulation as  \cite{arnold2002unified}. We will show that our formulation leads to a slightly modified version of previously proposed nonlinear formulations for symmetric IP (SIPG)  \cite{gudi2008hp}, and Bassi--Rebay \cite{bassi2005discontinuous}, while  we recover the  LDG formulation of \cite{bustinza2004local}. Here we are interested in these methods since, among the methods investigated in \cite{arnold2002unified},  these are the only ones that are stable and consistent for primal and adjoint solutions, and we might hope for extending these  properties to the corresponding nonlinear discretization. Also we prove that the approximate solution of the discontinuous Galerkin problem is well-posed; i.e, 
unique and stable, provided that the diffusion operator is strongly monotone and globally Lipschitz continuous. This is comparable to analysis of \cite{bustinza2004local} for LDG and somewhat similar to \cite{dolejvsi2008analysis} for IP; however our results cover different formulations. 
It is plausible that the same kind of analysis is applicable to cases which do not satisfy either strong monotonicity or global Lipschitz continuity as  in \cite{ortner2007discontinuous, gudi2007discontinuous, gudi2008hp}; nevertheless, we do not address this extension here and restrict our analysis to these two assumptions.
 
Secondly, we are going to provide an optimal error estimate for SIPG and Bassi--Rebay methods in terms of mesh size, both in energy norm and in the $L^2$-norm. Such estimates have previously been derived for LDG methods, in \cite{bustinza2004local} for monotone and globally Lipschtitz continuous problem, and in \cite{gudi2008hpmixed} for cases which neither of these two assumptions hold. For the SIPG method, error estimates have been presented in~\cite{gudi2008hp,  gudi2007discontinuous} and for incomplete penalty (IIPG) in \cite{ortner2007discontinuous}, for non-monotone and not globally Lipschitz operators. Also we have the results of \cite{dolejvsi2008analysis} for IIPG in case of monotone and globally Lipschitz operators.  Let us remark that since the formulations presented in \cite{
ortner2007discontinuous, dolejvsi2008analysis}  are adjoint inconsistent, they derived the error estimate only in the energy norm and did not deal with the corresponding adjoint problem.

{Besides the canonical approach of discretization we present here, the rigorous analysis of the second method of Bassi--Rebay is novel in the literature, as well as the different approach we adopted in the $L_2$ error estimate for an asymptotically adjoint consistent formulation. Furthermore, we present explicit conditions of stability  for all formulations. To the best knowledge of the authors this has not been done before for a nonlinear version of Bassi--Rebay methods and is similar to the explicit bounds for the SIPG method in \cite{shahbazi2005explicit, epshteyn2007estimation} and \cite{may2016spacetime}. {It is worth mentioning  that for the SIPG method, unlike the cited literature, our stability analysis remains valid for degenerate diffusion.} }
}

\ifdetail
\mz{Since our proposed methods in this paper, similar to  \cite{bustinza2004local} are inconsistent,  a new challenge  in our approach is handling the inconsistency terms appeared in the primal and dual formulation and proving their optimal convergence.  However, unlike \cite{bustinza2004local} which relies on the mean-value linearization between exact and DG solution, our adjoint equation comes from a linearization around the exact solution, like \cite{gudi2008hpmixed}. Also note that since we are concerned  with elliptic problems, the Gronwall lemma and the elliptic projection technique as in \cite{arnold1982interior, wheeler1978elliptic} and more recently \cite {kuvcera2010optimal} are not applicable. Another approach in the literature is \cite{abdulle2016error} which uses some projection for an special case of nonmonotone diffusion. }
\mrg{\mzdet{ as well as \cite{abdulle2011effect, abdulle2012priori} and the references cited therein. Note that \cite{abdulle2012priori, abdulle2011effect,  abdulle2016error} are for continuous finite element methods and considered the effect of the non-exact numerical integration on in the proofs.}}
\fi

\ifdetail
\mz{In addition to the argument presented in \cite{bustinza2004local} and \cite{gudi2008hpmixed}, here we present a discussion on the regularity requirement of the solution in order to obtain the $\Hm{2}$-regularity of the dual problem and the optimal error estimate consequently.}
\fi 

The structure of this paper is as  follows: In \S\ref{Sec::General} we provide a short review on the quasi-linear elliptic problems and the properties of the diffusion operator. In \S\ref{Sec::Prelim} we review some approximation results as well as the properties of triangulation of the computational domain. Section~\ref{Sec::DG} deals with our canonical approach for obtaining different DG formulation in the primal form, and later in \S\ref{Sec::Consst}, we analyze the consistency and adjoint consistency of the presented formulations. The sections \ref{Sec::stab} and \ref{Sec::Exist} are devoted to the stability and uniqueness of the DG approximate solution, respectively. Note that before reaching section~\ref{Sec::Exist} we do not require the monotonicity of the operator and the stability result is valid for more general problems. Section~\ref{Sec::Error} presents the optimal error convergence estimate in both energy and $\Ltwo$ norms.

\section{Quasi-linear elliptic problems}
\label{Sec::General}
We consider the following quasilinear elliptic problem 
\begin{align}
\label{Eq::origPDE}
-\divg  \diff(x, \uu, \nabla \uu) &= f,  &&\text{ in } \dom, \\
\uu & = \uu_D,  &&\text{ on }   \partial \dom,
\end{align}
where $f \in \Ltwo(\dom)$ and $\uu_D \in \Hm{{1/2}}(\partial \dom)$. For simplicity we will set $\uu_D \equiv 0$ in later analysis.
Also $\dom$ is a bounded and simply connected domain  in $\Real^2$. 
\ifdetail \mrg{\mzdet{Safer to stay in $2D$, the estimates we have in \cite{gudi2008hp} and \cite{bustinza2004local} is for 2D. See Schwab for 3D}}\fi
For brevity one might use the notation $\diff(\cdot, \bm{\zeta})  = \diff(\cdot, \uu, \nabla \uu) $ where $\bm{\zeta} \in \Real^3$, such that  $\zeta_0 \equiv \uu$ and $\zeta_i \equiv \uu_{x_i}, i= 1, 2$. Moreover, by $\diff_u$ and $\diffz := \big[ \frac{\partial \diff_i(\cdot, \bm{\zeta})}{\partial \zeta_j}\big]_{i, j= 1, 2}$, we denote the derivative of  $\diff(x, \uu, \nabla \uu)$ with respect to its second and third arguments.

In the general theory  of nonlinear elliptic problems (see \cite{nevcas1983introduction, vzenivsek1990nonlinear}), it is usually assumed that the function $\diff(\cdot, \bm{\zeta})=(\sdiff_1(\cdot, \bm{\zeta}), \sdiff_2(\cdot, \bm{\zeta}))$ satisfies some conditions:	
\begin{enumerate}[label= A({\roman*})]
\item 
\ifdetail \mgg{In light of assumption A.3: Do we need A.1 at all?} \fi
The functions $\sdiff_i(x, \bm{\zeta}), i = 1, 2$ are continuous in $\dom \times \Real^{3}$ and satisfy the following growth condition; 
\begin{equation}
\vert \sdiff_i(x, \bm{\zeta}) \vert \leq c_1 \Big( 1  + \sum_{i=0}^2 \vert \zeta_i \vert \Big) + \vert \phi_i(x) \vert, \qquad \forall \bm{\zeta}  \in \Real^{3}, \forall x \in \dom,
\end{equation}
where $c_1 > 0 $ and $\phi_i \in \Ltwo(\dom)$, for $i = 1, 2$.
%
\label{Ass::1}
\item There exist two constants $0 < \lambda \leq \Lambda < \infty$ such that for all $\bm{\zeta}, \bm{\psi} \in \Real^3$
\begin{equation}
\label{Eq::ass_mon_bnd}
0 < \lambda \sum_{k = 1}^2 \psi_k^2 \leq \sum_{\substack{j = 1 \\ k = 0}}^{2} \dfrac{\partial \sdiff_j(\cdot, \bm{\zeta})}{\partial \zeta_k} \psi_j \psi_k  \leq \Lambda \sum_{k = 1}^2 \psi_k^2.
\end{equation}
\ifdetail \mrg{\mzdet{, in the special case $\bm{\psi} = (0, \bm{\psi}') \in \Real^3$ where $\bm{\psi}' \in \Real^2$,}} \fi
We also might consider a relaxed version of \eqref{Eq::ass_mon_bnd} as the following 
\begin{equation}
\label{Eq::ass_stab_bnd}
0 < \lambda \sum_{k = 1}^2 \psi_k^2 \leq \sum_{\substack{j = 1 \\ k = 1}}^{2} \dfrac{\partial \sdiff_j(\cdot, \bm{\zeta})}{\partial \zeta_k} \psi_j \psi_k  \leq \Lambda \sum_{k = 1}^2 \psi_k^2.
\end{equation}

For simplicity in the later analysis we assume that $\diffz$ is symmetric. Then we can interpret \eqref{Eq::ass_stab_bnd} as imposition of lower and upper bound on the eigenvalues of $\diffz$; $\lambda$ and $\Lambda$, respectively.
\label{Ass::2}
\item The functions $\sdiff_i(x, \bm{\zeta}), i=1, 2$ are in $\mathcal{C}_b^2(\dom   \times \Real^3)$, i.e., they are twice continuously differentiable functions
with all the derivatives through second order being bounded.
\label{Ass::3}
\end{enumerate}

Note that \ref{Ass::1} is required to guarantee the meaningfulness of the of the definition of the discrete problem. On the other hand the assumptions \ref{Ass::2} and \ref{Ass::3} need to be valid only for the arguments provided in sections~\ref{Sec::Exist} and \ref{Sec::Error} for uniqueness of the discrete solution and error estimate, respectively. Furthermore, the relaxed version of assumption~\ref{Ass::2} in \eqref{Eq::ass_stab_bnd} is required in the stability analysis in section~\ref{Sec::stab}.

Let us note the following important definition for nonlinear elliptic PDEs, the so-called \emph{strong monotonicity} and \emph{global Lipschitz continuity} property of the diffusion operator (as \cite{feistauer1986finite}, \cite[p.~476]{zeidler2013nonlinear}):
\ifdetail or \mz{ \cite[p.~??]{nevcas1983introduction}}) \fi
\begin{definition}
\label{Def::Monotone}
We say a diffusion operator $\diff(x, \uu, \nabla \uu)$ satisfies strong  monotonicity property, if there exists a constant $\Csm > 0$ such that 
\begin{equation}
\label{Eq::Strong}
\inprods{\big(\diff(x, \xi', \bm{\eta}') - \diff(x, \xi, \bm{\eta}) \big)}{(\bm{\eta}' - \bm{\eta})} \geq \Csm \vert \bm{\eta} - \bm{\eta}' \vert^2 
\end{equation} 
for all $\xi, \xi' \in \Real$ and $\bm{\eta}, \bm{\eta}' \in \Real^2$. Moreover we call it a globally Lipschitz operator if there exists a constant $\Clc < \infty$ such that 
\begin{equation}
\label{Eq::Lip}
 \diff(x, \xi', \bm{\eta}') - \diff(x, \xi, \bm{\eta})  \leq \Clc \big[\vert \xi - \xi' \vert^2 +  \vert \bm{\eta} - \bm{\eta}' \vert^2 \big]^{1/2}, 
\end{equation} 
for all $\xi, \xi' \in \Real$ and $\bm{\eta}, \bm{\eta}' \in \Real^2$.
\end{definition}

One can show that if the left and right inequalities in \eqref{Eq::ass_mon_bnd} are satisfied, the diffusion operator $\diff$ is strongly monotone and globally Lipschitz continuous in the sense of Definition~\ref{Def::Monotone}. We skip the proof here and refer to \cite[Lemma 2.1]{barrett1994quasi} or \cite[Lemma 4]{dolejvsi2008analysis}.

\ifdetail
Let also simplify these two properties in a more useful form for the rest of paper:\mgg{do we really need to re-state these definitions?}
\begin{enumerate}
\item Strong monotonicity property
\begin{equation}
\Big( \diff(x, \vv, \nabla \vv) - \diff(x, \ww, \nabla \ww)   \Big) \cdot (\nabla \vv - \nabla \ww) \geq \Csm \vert \nabla \vv - \nabla \ww \vert^2
\end{equation}
\item Global Lipschitz continuity
\begin{equation}
\vert \diff(x, \vv, \nabla \vv) - \diff(x, \vv', \nabla  \vv')  \vert \leq \Clc \Big[ \vert \vv -  \vv'  \vert^2 + \vert  \nabla \vv - \nabla \vv'  \vert^2 \Big]^{1/2}
\end{equation}
\end{enumerate}
\fi

As examples of problems settled in the category of \eqref{Eq::origPDE} we have
\begin{enumerate}[label= ({\alph*})]
\item Newtonian flow model
\begin{equation}
\diff(x, \uu, \nabla \uu) = \diff(x, \uu) \nabla \uu
\end{equation}
\item Non-Newtonian flow model
\begin{equation}
\diff(x, \uu, \nabla \uu) = \diff(x, \vert \nabla \uu \vert) \nabla \uu
\end{equation}
with an specific case of mean curvature flow
\begin{equation}
\diff(x, \uu, \nabla \uu) = \dfrac{\nabla \uu}{(1 + \vert \nabla \uu \vert)^{1/2}}
\end{equation}
\end{enumerate}

We remark that for the case (b), following \cite{bustinza2004local, houston2005discontinuous}, we typically assume the monotonicity of the diffusion operator (see Definition~\ref{Def::Monotone}), while the diffusion operator in case (a) is not usually a monotone operator  (only under very restrictive assumptions, see \cite{abdulle2012priori}). Since the presentation of the primal form of DG formulations in section~\ref{Sec::DG} and stability result in section~\ref{Sec::stab} are not affected by the monotonicity property, we do not exclude this case now. But in obtaining a priori error estimate in section~\ref{Sec::Error}, we will only consider monotone cases and one can refer to \cite{ortner2007discontinuous}, \cite{gudi2007discontinuous} or \cite{gudi2008hp} for a priori estimate for non-monotone cases.  One may check that the mean curvature flow example satisfies the assumption~\ref{Ass::2}; hence it is strongly monotone and globally Lipschitz continuous, see \cite[section 5]{gudi2008hp}.

\section{Preliminaries}
\label{Sec::Prelim}
First let us set a notation convention and suppress  the dependence of the diffusion operator on the space variable $x$; henceforth we write $\diff(\vv, \bm{z})$ instead of $\diff(x, \vv, \bm{z})$, while we still allow explicit dependence on $x$.

Now, as a tool that we will use later in sections~\ref{Sec::stab} and \ref{Sec::Error}, let us consider the following integral form of Taylor's formula; for $\uu, \vv \in \Real$ and $\bm{z}, \bm{w} \in \Real^2$ one has
\begin{align}
\label{Eq::Taylor}
\diff(\vv, \bm{z}) - \diff(\uu, \bm{w}) &= \diffu(u, \bm{w})(v - u) + \diffz(u, \bm{w})(\bm{z} - \bm{w}) + R_{\diff}(u - v, \bm{w} - \bm{z})   \nonumber \\
& = \difftu(u, \bm{w})(v - u) + \difft(u, \bm{w})(\bm{z} - \bm{w})
\end{align}
where $R_{\diff}(u - v, \bm{w} - \bm{z}) = (R _{\diff_1}(u - v, \bm{w} - \bm{z}),R _{\diff_2}(u - v, \bm{w} - \bm{z}))$ is defined as
\begin{equation}
R _{\diff_i} = \tilde{\diff}_{\uu \uu} (\uu - \vv)^2 + (\bm{w} - \bm{z})\tr \tilde{\diff}_{\bm{z} \bm{z} }(\uu, \bm{w}) (\bm{w} - \bm{z}) + 2 \tilde{\diff}_{\uu  \bm{z} }(\uu, \bm{w}) (\bm{z} - \bm{w} ) (\uu - \vv).
\end{equation}
for $i = 1, 2$.
Moreover we define $\difftu, \difft, \tilde{\diff}_{\uu \uu}, \tilde{\diff}_{\uu \bm{z}}, \tilde{\diff}_{\bm{z} \bm{z}}$ as 
\begin{align*}
\difftu (\uu, \bm{w}) &=\int_0^1 \diffu(\vv(t), \bm{z}(t)) \df{t}, && \difft(\uu, \bm{w}) = \int_0^1 \diff_{\bm{z}}(\vv(t), \bm{z}(t)) \df{t} \\
 \tilde{\diff}_{\uu \uu}(\uu, \bm{w}) &= \int_0^1 (1-t) \diff_\uu(\vv(t), \bm{z}(t)) \df{t},  &&\tilde{\diff}_{\uu \bm{z}}(\uu, \bm{w})  = \int_0^1 (1-t) \diff_\uu(\vv(t), \bm{z}(t)) \df{t}, \\ \tilde{\diff}_{\bm{z} \bm{z}}(\uu, \bm{w}) &= \int_0^1 (1-t) \diff_\uu(\vv(t), \bm{z}(t)) \df{t}
\end{align*}
where $\vv(t) = \uu + t (\vv - \uu)$ and $\bm{z}(t) = \bm{w} + t (\bm{z} - \bm{w})$. 

Seeking clarity, sometimes we denote $\difftu, \difft, \tilde{\diff}_{\uu \uu}, \tilde{\diff}_{\uu \bm{z}}, \tilde{\diff}_{\bm{z} \bm{z}}$ by all four arguments, \break e.g. $\difftu(\uu, \bm{w}, \vv, \bm{z})$ instead of $\difftu(\uu, \bm{w})$.
 
\subsection{Triangulation and finite element space}
Here we are going to consider the boundary of the domain $\dom$ sufficiently smooth (in order to apply the duality argument, see section~\ref{Subsec::adjconsis}), e.g. a convex polygon. Then we consider a shape-regular triangulation on $\dom$ as $\triag = \{ \telem \}$ composed of (non-overlapping) triangular or rectangular elements (with possible hanging nodes) and $\hk$ is the diameter of each $\telem \in \triag$. Also we define $\h := \max_{\telem \in \triag} \hk$ and $\no_{\telem}$ is the outward normal to $\partial \telem$. In the following we assume that $\triag$ is of \emph{bounded variation}, that is, there exists a constant $ l > 1 $ such that
\begin{equation}
\LL^{-1} \leq \dfrac{\h_{\telem}}{\h_{\telem'}} < \LL,
\end{equation}
where $\telem, \telem' \in \triag$ share an edge. 
This bounded variation property means that there is an upper bound for the number of neighboring elements of each $\telem \in \triag$, denoted by $\non$. In case that $\triag$ has no hanging nodes $\non = 3$ and $\non = 4$ for triangular and rectangular elements, respectively.
\ifdetail \mrg{This is the same as Reusken Definition 3.3.2}\fi
We also need the following \emph{quasi-uniformity} property of the mesh in the $\Ltwo$ error analysis in section~\ref{Subsec::Ltwo}, that is
\begin{equation}
\label{Eq::quasi_unif}
\h \leq C \hk, \qquad \forall \telem \in \triag.
\end{equation}
We denote the \emph{skeleton} of the triangulation $\triag$, i.e. the set  of all edges of $\telem \in \triag$, by $\sklt$. Also we denote the set of boundary and interior edges of $\triag$ by $\skltb$ and $\sklti$, respectively, and the length of edge $\face$   by $\he$.
Following \cite{arnold2002unified}, let us fix some definitions for the jumps and average of the discontinuous functions on the skeleton $\sklt$. Let us set the trace values as $\ww_{\telem, \face} = \ww_\telem \vert_\face $. For any interior edge $\face \in \sklti$, where $\face$ is the common edge of $\telem, \telem'\in \triag$, and for all $\ww \in \prod_{\telem \in \triag} \Ltwo(\partial \telem)$, we define
\begin{equation}
\avg{\ww} = \dfrac{1}{2}(\ww_{\telem, \face}  + \ww_{\telem', \face}), \qquad \jump{\ww} = {\ww_{\telem, \face} }{\no_\telem} + {\ww_{\telem', \face}}{\no_{\telem'}}
\end{equation}
and similarly for all $\taub \in \prod_{\telem \in \triag} [ \Ltwo(\partial \telem)]^2$
\begin{equation}
\avg{\taub} = \dfrac{1}{2}(\taub_{\telem, \face} + \taub_{\telem', \face}), \qquad \jump{\taub} = \inprods{\taub_{\telem, \face}}{\no_{\telem}} + \inprods{\taub_{\telem', \face}}{\no_{\telem'}}.
\end{equation}
For any boundary edge $\face \in \skltb$ we define 
\begin{equation}
\jump{\ww} = \ww_{\telem, \face}  \no_{\telem}, \qquad \avg{\taub} = \taub_{\telem, \face} ,
\end{equation}
for all $\ww \in \prod_{\telem \in \triag} \Ltwo(\partial \telem)$ and $\taub \in \prod_{\telem \in \triag} [ \Ltwo(\partial \telem)]^2$.

Moreover,  let us consider the following broken Sobolev space on the triangulation $\triag$; for $1 \leq r < \infty$
\begin{equation}
\Wmp{r}{\smth}(\dom, \triag) = \{ \vv \in \Lp{r}: \vv \vert_{\telem} \in \Wmp{r}{\smth}, \forall \telem \in \triag \},
\end{equation}
with the corresponding norm and seminorm
\begin{equation}
\Vert \vv \Vert_{\Wmp{r}{\smth}(\dom, \triag)} = \Big( \sumk \Vert \vv \Vert^r_{\Wmp{r}{\smth}(\telem)} \Big)^{1/r}, \qquad \seminorm{\vv}_{\Wmp{r}{\smth}(\dom, \triag)} = \Big( \sumk \seminorm{ \vv}^r_{\Wmp{r}{\smth}(\telem)} \Big)^{1/r},
\end{equation}
and for the case $r = \infty$ the associated norm and seminorm are defined as
\begin{equation}
\Vert \vv \Vert_{\Wmp{\infty}{\smth}(\dom, \triag)} = \max_{\telem \in \triag } \Vert \vv \Vert_{\Wmp{\infty}{\smth}(\telem)} , \qquad \seminorm{\vv}_{\Wmp{\infty}{\smth}(\dom, \triag)} =  \max_{\telem \in \triag } \ \seminorm{ \vv}_{\Wmp{\infty}{\smth}(\telem)},
\end{equation}
when $\Vert \cdot \Vert_{\Wmp{r}{\smth}(\telem)}$ and  $\vert \cdot \vert_{\Wmp{r}{\smth}(\telem)}$ are the standard Sobolev norms on $\telem$. Also we denote $\Wmp{2}{\smth}$ as $\Hm{\smth}$ by tradition. 

Now let define the following finite dimensional approximation spaces
\begin{align}
\Vsp &:= \{ \uh \in \Ltwo(\dom) : \uh \vert_\telem \in \polp(\telem), \quad \forall \telem \in \triag \}, \\
\Sigsp &:= \{ \teth \in [\Ltwo(\dom)]^2 : \teth \vert_\telem \in [\polaltp(\telem)]^2, \quad \forall \telem \in \triag \},
\end{align}
with $\poldeg \geq 1$ and $\altpoldeg = \poldeg$ or $\altpoldeg = \poldeg - 1$ (To satisfy the inclusion property $\nabla \Vsp \subset \Sigsp$ as \cite{arnold2002unified}.) 
\ifdetail
 \mrg{\mzdet{For a recently proposed choice $\altpoldeg = \poldeg + 1$ and its properties see \cite{john2016stable}.}}.
\fi
 Here by $\polp(\telem)$ we denote the space of the polynomials of total degree $\poldeg$ on $\Real^2$ and restricted to $\telem$.

Moreover, let us define two lifting operator $\liftR: [\Ltwo(\sklt)]^2 \to \Sigsp$ and $ \liftL: \Ltwo(\sklti) \to \Sigsp$ as
\ifdetail \mgg{From here on, we need to decide whether we want $\intdom$ or $\int_{\triag}$. I guess if we use $\intdom$ we need to stick with polygonal domains. (We mention this above, so it's ok.)} \fi
\begin{align}
\label{Eq::lift_LR}
\intdom \inprods{\liftR(\varphi)}{\bm{\tau}} \dx= - \sume \inted \inprods{\varphi}{\avg{\bm{\tau}}} \ds, \qquad \intdom \inprods{\liftL(\varphi)}{\bm{\tau}} \dx= - \sumei \inted  \varphi \jump{\bm{\tau}} \ds,
\end{align}
for all  $ \bm{\tau} \in \Sigsp$. Using the Riesz representation theorem one can prove the existence and uniqueness of the lifting operators introduced by \eqref{Eq::lift_LR} (see e.g., \cite[Lemma 3.3]{bustinza2004local}).
Also we define an edge-wise version of right and left lifting operators as $\liftRe: [\Lp{2}(e)]^d \to \Sigsp$ and $\liftLe: \Lp{2}(e) \to \Sigsp$, such that
\begin{align}
\intdom \inprods{\liftRe(\varphi)}{\bm{\tau}} \dx &= - \inted \inprods{\varphi}{\avg{\bm{\tau}}} \ds, &&\forall \taub \in \Sigsp, \forall \face \in \sklt,\\
\intdom \inprods{\liftLe(\varphi)}{\bm{\tau}} \dx &= - \inted  \varphi \jump{\bm{\tau}} \ds, &&\forall \taub \in \Sigsp, \forall \face \in \sklti
\end{align}
for all edges $\face \in \sklt$. Also by noting that $\liftR(\varphi) = \sume \liftRe(\varphi) $ and, by applying Cauchy-Schwarz inequality and taking the norm over $\triag$, one has
\begin{equation}
\label{Eq::liftcauchy}
\Vert \liftR(\varphi) \Vert^2_{\Ltwo(\dom)} = \Vert \sume \liftRe(\varphi) \Vert^2_{\Ltwo(\dom)}  \leq \non   \sume \Vert  \liftRe(\varphi) \Vert^2_{\Ltwo(\dom)}.
\end{equation}
Furthermore, one might note that for any $\face \in \sklti$ and $\face \subset \partial \telem, \telem \in \triag$ it holds $\liftLe(\varphi) = 2 \liftRe(\varphi \no_\telem)$ and consequently
\begin{equation}
\label{Eq::left_lift_bound}
\Vert \liftL(\varphi) \Vert^2_{\Ltwo(\dom)} = \Vert \sume \liftLe(\varphi) \Vert^2_{\Ltwo(\dom)}  \leq \non   \sume \Vert  \liftLe(\varphi) \Vert^2_{\Ltwo(\dom)} \leq 4 \non   \sume \Vert  \liftRe(\varphi \no_\telem) \Vert^2_{\Ltwo(\dom)}.
\end{equation}
using Cauchy-Schwarz inequality.

Also let us introduce the space $\Vh := \Vsp + \Hm{2}(\dom) \bigcap  \Hm{1}_0(\dom)$ and the corresponding  energy {norm} $\ennorm{\cdot}: \Vh \to \Real$ as
\ifdetail
\mrg{As long as $\vert \partial \dom _D\vert \neq 0$ defines a norm on $\Vh$? I don't think do unless $g$ is a polynomial function. Otherwise we may need to keep the boundary in $\Vert \alpha u \Vert_{\partial_D} $ }.
\fi
\begin{equation}
\label{Eq::energynorm}
	\ennorm{\vv}^2 := \seminorm{\vv}^2_{1, \dom} + \sume \normgen{\liftR^\face(\jump{\vh})}^2_{\Ltwo(\dom)}, \qquad \forall \vv \in \Vh
\end{equation}
as well as the following seminorm $ \liftnorm{\cdot}: \Vh \to \Real$ as 
\begin{equation}
\label{Eq::lift_norm}
\liftnorm{\vv}^2 := \sume \normgen{\liftR^\face(\jump{\vv})}^2_{\Ltwo(\dom)}, \qquad \forall \vv \in \Vh.
\end{equation}
From the structure of \eqref{Eq::lift_norm}, using \eqref{Eq::liftcauchy} and \eqref{Eq::left_lift_bound} reads, there exists $C_s < \infty$ such that
\begin{equation}
\label{Eq::liftnorm_bound}
\Vert \liftR(\vv) \Vert^2_{\Ltwo(\dom)} +  \Vert \liftL(\vv) \Vert^2_{\Ltwo(\dom)} \leq C_s  \liftnorm{\vv}^2
\end{equation}
for all $\vv \in \Vh$.
Moreover, we have the following estimate  from \cite[Lemma 2]{brezzi2000discontinuous}
\begin{lemma}
\label{Lemma::lift_jump}
There exist two positive constants $\Cr, \CR >0$, such that for all $\face \in \sklt$ we have 
\begin{equation}
\label{Eq::lift_jump}
\Cr \he^{-1/2} \Vert \jump{\ww} \Vert_{\Ltwo(\face)} \leq \Vert \liftR^\face (\jump{\ww}) \Vert_{\Ltwo(\dom)}\leq 
\CR \he^{-1/2} \Vert \jump{\ww} \Vert_{\Ltwo(\face)}, \qquad \forall \ww \in \Vh,
\end{equation}
where the constants are $\h$ independent and only depend on the minimum angle of the triangles and the polynomial degree $\poldeg$.
\end{lemma}
\begin{proof}
 The original proof in~\cite{brezzi2000discontinuous} was presented for $\wh \in \Vsp$, while in \cite{arnold2002unified} the extended version to $\Vh$ was given, exploiting the Sobolev embedding to deduce $\Vh \setminus \Vsp \subset \Hm{2} \subset \mathcal{C}(\dom)$ in $\Real^2$. Then  \eqref{Eq::lift_jump} trivially holds for all $\ww \in \Vh \setminus \Vsp$.  For the details on the explicit value of $\Cr$ and $\Cr$ we refer to \cite{brezzi2000discontinuous} and \cite{warburton2003constants}.
\end{proof}

In order to see the relation between $\ennorm{\cdot}$ and $\ltwonrm{\cdot}{\dom}$,  let us remark the following relation from \cite[Lemma 2.1]{arnold1982interior}, for all $\vv \in \Hm{1}(\dom,\triag)$
\begin{equation}
\label{Eq::Arnold_jumpnrm}
\Vert \vv \Vert^2_{\Ltwo(\dom)} \leq C(\dom, \triag) \Big[ \Vert \nabla \vv \Vert^2_{\Ltwo(\dom)} +\sume \he^{-1} \Vert \jump{\vv} \Vert^2_{\Ltwo(\face)} \Big],
\end{equation}
which holds when $\dom$ is convex. For a similar result on non-convex domains we refer to \cite{bustinza2004local, gopalakrishnan2003multilevel, rusten1996interior}.

Now, restricted to $\vv \in \Vh$ and by using \eqref{Eq::lift_jump}, \eqref{Eq::Arnold_jumpnrm} reduces to the following Poincar\'e-type inequality
\begin{equation}
\label{Eq::energynrm_eq}
 \Vert \vh \Vert^2_{\Ltwo(\dom)} \leq \Cen \ennorm{\vh}^2,
\end{equation}
with some $\Cen < \infty$, and by the definition of the energy norm \eqref{Eq::energynorm} one can write
\begin{equation}
\label{Eq::energynrm_eq_H1}
\Vert \vh \Vert^2_{\Hm{1}(\dom, \triag)} \leq (\Cen+1) \ennorm{\vh}^2.
\end{equation}
Also we will need the following inverse inequality
\begin{lemma}
Let consider $\vh \in \Vsp$, then for $r \geq 2$ there exists a constant $\Cinv > 0 $ such that
\begin{equation}
\label{Eq::Inverse}
\Vert \vh \Vert_{\Lp{r}(\telem)} \leq \Cinv \hk^{2/r - 1} \Vert \vh \Vert_{\Lp{2}(\telem)}.
\end{equation}
\end{lemma}
The proof of this lemma can be found, e.g., in \cite[p.~140]{ciarlet1978finite} and we skip it.
\subsection{Approximation properties}

\label{Secsec::approx}
First, let state  some approximation properties in the next two lemmas
\begin{lemma}
\label{Lem::hp_approx}\cite [Lemma 2.1]{gudi2008hp}
For $\phi \in \Hm{\smth}(\telem)$, there exists a positive constant $\CA$, depending on $\smth$ and $\poldeg$, but independent of $\phi$ and $\hk$, and
 a sequence $\phi_{\hk} \in \polsp^{\poldeg}(\telem)$, $\poldeg = 1, 2, \cdots$ such that
\begin{enumerate}[label= ({\roman*})]
\ifdetail 
\mzdet{
\item for any $ 0 \leq l \leq \smth$
\begin{equation}
\Vert \phi - \phi_{\hk} \Vert_{\Hm{l}(\telem)} \leq \CA {\hk}^{\mu - l} \Vert \phi \Vert_{\Hm{\smth(\telem)}}
\end{equation}
}
\fi
\item for any $ \smth \geq  l + \frac{1}{2} $
\begin{equation}
\Vert \phi - \phi_{\hk} \Vert_{\Hm{l}(\face)} \leq \CA  {\hk}^{\mu  -l - 1/2} \Vert \phi \Vert_{\Hm{\smth(\telem)}}
\end{equation}
\item for any $ 0 \leq l \leq s - 1 + \frac{2}{r} $
\begin{equation}
\Vert \phi - \phi_{\hk} \Vert_{\Wmp{r}{l}(\telem)} \leq \CA {\hk}^{\mu -l -  1 + 2/r} \Vert \phi \Vert_{\Hm{\smth(\telem)}}
\end{equation} 
\end{enumerate}
where $\mu = \min(\poldeg+1, s)$.
\end{lemma}

For the proof of this lemma we refer to \cite[Lemma 2.1]{gudi2008hp} and the references cited therein.
\ifdetail\mrg{\mzdet{all the papers are on the $p$ version, probably I can find more related to $\h$}} \fi

Using Lemma~\ref{Lem::hp_approx} and the properties of the energy norm \eqref{Eq::energynorm}, we present the following approximation result 
\begin{lemma}
\label{Lem::Approx}
For $\forall \phi \in \Vh$ there exists a constant $\CA' > 0$ independent of $\h$ and $\phi$, and a mapping $\projh:\Vh \to \Vsp $ such that
\begin{equation}
\ennorm{ \phi - \projh \phi} \leq \CA' \Big( \sumk \hk^{2(\mu - 1)} \seminorm{\phi}^2_{\Hm{\smth}(\telem)} \Big)^{1/2},
 \end{equation}
 where $\mu = \min(\poldeg+1,  \smth)$.
\end{lemma}
\begin{proof}
The proof exploits Lemmas~\ref{Lemma::lift_jump} and \ref{Lem::hp_approx} and the definition of $\Vh$ which provides continuity of $\ww \in \Vh \setminus \Vsp$. Then the proof follows the same lines as \cite[section 4.3.]{arnold2002unified}.
\end{proof}

Note that in our analysis in this work, we do not need to specify the explicit type of projection $\projh$. Hence we leave it undefined with merely assuming that it satisfies the approximation property described in Lemmas~\ref{Lem::hp_approx} and \ref{Lem::Approx}. We refer to \cite{bustinza2004local} and \cite{gudi2008hp} for explicit examples of such projections. Let us remark that the Galerkin $[\Ltwo]^2$ projection that we introduce in \eqref{Eq::GalProj} also satisfies these approximation properties. 

\section{Discontinuous Galerkin formulation}
\label{Sec::DG}
Here we follow the formulation presented in \cite{bustinza2004local} (also see \cite{cockburn1999some} and \cite{perugia2002hp}). Let us start with writing the nonlinear elliptic problem \eqref{Eq::origPDE} as a system of first order nonlinear PDEs in terms of new variables  $(\sig, \tet,  \uu)$:
\begin{align}
\label{Eq::PDEfirst}
- \nabla \cdot \sig &= f, &&\text{in } \dom, \\
\sig &= \diff(\uu, \tet), &&\text{in } \dom, \\
\tet &= \nabla \uu,  &&\text{in } \dom, \\
\uu &= 0,   &&\text{on } \partial \dom. \label{eq:bdyCond}
\end{align}

Our goal is to approximate the exact solution $(\sig, \tet,  \uu)$ by discrete functions $(\sigh, \teth, \uh)$ in the finite element space $\Sigsp \times \Sigsp \times \Vsp$. The weak formulation can be written as
\begin{align}
\label{Eq::weak_sig}
&\intdom \inprods{\bm{a}(\uh, \teth)}{ \zeth} \dx = \intdom \inprods{\sigh}{\zeth} \dx, \quad &\forall
 \zeth \in \Sigsp, \\
\label{Eq::weak_tet}
&\intdom \inprods{\teth}{\tauh} \dx + \intdom {\uh}{ (\divh \tauh)} \dx = \sumk \intdk \uflux \inprods{\tauh}{\no} \ds ,  &\forall \tauh \in \Sigsp, \\
\label{Eq::weak_u}
& \intdom \inprods{\sigh}{\gradh \vh} \df{x} - \sumk \intdk {\vh}{(\inprods{\sflux}{\no})} \ds = \intdom f \vh\dx , \quad &\forall \vh \in \Vsp.
\end{align}
Here $\gradh$ and $\divh$ are the element-wise version of the gradient and divergence operator, respectively. 
\ifdetail{\mrg{\mzdet{write the space and general form of numerical fluxes, in which space?}}\fi

This \emph{flux} formulation is complete but the definition of the numerical fluxes $\uflux$ and $\sflux$ which depends on $(\sigh, \teth, \uh)$, and needs to be designed carefully. \ifdetail{such that the method has good well-posedness and compatibility properties.}\fi We postpone the explicit definition of these numerical fluxes till the next section, where we  present the equivalent \emph{primal} formulation; i.e., a formulation which has $\uh$ as its only unknown.

The only requirement we impose here is that the $\uflux$ should be independent of $\teth$ and  $\sigh$, i.e., $\uflux = \uflux(\uh)$, which provides the availability of the primal formulation. All the formulations we are going to present have this local property. For examples of other forms of non-local formulation and their analysis we refer to \cite{yadav2013superconvergent}, \cite{gudi2008hpmixed} and \cite{gatica2002solvability}.
\subsection{Primal formulation} 
In order to obtain the primal formulation we need to solve the unknowns $\sigh$ and $\teth$ in terms of $\uh$. In the first step, by using \eqref{Eq::weak_sig}, one can solve for $\sigh$ as the Galerkin $[\Ltwo(\dom)]^2$ projection,
\begin{equation}
\label{Eq::sig_weak_proj}
\sigh = \galproj\big(\diff(\uh, \teth)\big),
\end{equation}
where  $\galproj: [\Ltwo(\dom)]^2 \to \Sigsp $ has the following property; for all $\bm{\xi} \in [\Ltwo(\dom)]^2$
\begin{equation}
\label{Eq::GalProj}
\sumk \intk \inprods{\bm{\xi}}{\taub} \dx = \sumk \intk \inprods{\galproj(\bm{\xi})}{\taub} \dx, \qquad \forall \taub \in \Sigsp.
\end{equation}
Moreover, let us remark the following identity; for any $\vv \in \prod_{\telem \in \triag} \Ltwo(\partial \telem)$ and $\bm{\xi} \in \prod_{\telem \in \triag} [\Ltwo(\partial \telem)]^2$ the following holds 
\begin{equation}
\label{Eq::jump_avg_relation}
\sumk \intdk \vv \inprods{\bm{\xi}}{\no}  \ds =  \sumei \inted \avg{\vv} \jump{\bm{\xi}} \ds +  \sume \inted \inprods{\jump{\vv}}{ \avg{\bm{\xi}}} \ds.
\end{equation} 
Applying \eqref{Eq::jump_avg_relation} in \eqref{Eq::weak_tet} and \eqref{Eq::weak_u}, one can write
\begin{align*}
&\intdom \inprods{\teth}{\tauh} \dx - \intdom \inprods{\gradh \uh}{\tauh} \dx + \sumei \inted \avg{\uh - \uflux} \jump{\tauh} \ds + \sume \inted \inprods{\jump{\uh- \uflux}}{\avg{\tauh}} \ds  = 0, \\
&\intdom \inprods{\sigh}{\gradh \vh} \dx - \sume \inted \inprods{\avg{\sflux}}{\jump{\vh}} \ds -  \sumei \inted {\jump{\sflux}}{\avg{\vh}} \ds  = \intdom f \vh \dx.
\end{align*} 	
Now let us require the numerical fluxes $\uflux$ and $\sflux$ to be conservative, which leads to 
\begin{equation}
\label{Eq::Flux_conservation}
\jump{\uflux} = 0, \qquad \avg{\uflux} = \uflux, \qquad \jump{\sflux} =0, \qquad \avg{\sflux} =\sflux,
\end{equation}	
on any $\face \in \sklti$. Also in accordance with the Dirichlet boundary condition  \eqref{eq:bdyCond}, we set $\uflux = 0$ on $\face \in \skltb$.
Using \eqref{Eq::Flux_conservation}, one can simplify the weak formulation as
\begin{align}
&\intdom \inprods{\teth}{\tauh} \dx - \intdom \inprods{\gradh \uh}{\tauh} \dx + \sumei \inted (\avg{\uh} - \uflux) \jump{\tauh} \ds + \sume \inted \inprods{\jump{\uh}}{\avg{\tauh}} \ds = 0, \\
&\intdom \inprods{\sigh}{\gradh \vh} \dx - \sume \inted \inprods{\sflux}{\jump{\vh}} \ds  = \intdom f \vh \dx.
\label{Eq::FinalSecond}
\end{align} 
Using the definition of lifting operator \eqref{Eq::lift_LR} and \eqref{Eq::sig_weak_proj}, we arrive at
\begin{align}
\teth &= \gradh \uh + \liftR(\jump{\uh }) + \liftL(\avg{\uh} - \uflux), 
\label{Eq::teteq}\\
\sigh &= \galproj\big(\diff(\uh, \gradh \uh + \liftR(\jump{\uh }) + \liftL(\avg{\uh} - \uflux)) \big).
\label{Eq::sigeq}
\end{align}
Now $(\teth, \sigh)$ can be solved locally in terms of $\uh$ by inserting \eqref{Eq::teteq} and \eqref{Eq::sigeq} in \eqref{Eq::FinalSecond}. Then one can obtain the primal formulation as the following; find $\uh \in  \Vsp$ such that
\begin{equation}
\label{Eq::primal}
\difffrm(\uh, \vh) = \rhsfrm(\vh), \quad \forall \vh \in \Vsp,
\end{equation}
where $\displaystyle \rhsfrm(\vh) = \intdom f \vh \dx$, and
\begin{equation}
	\label{eq:generalPrimal}
	\difffrm(\uh, \vh) =\intdom \inprods{\diff\big(\uh, \gradh \uh + \liftR(\jump{\uh }) + \liftL(\avg{\uh} - \uflux) \big)}{\gradh \vh} \dx - \sume \inted \inprods{\sflux}{\jump{\vh}} \ds  .
\end{equation}

 Note that in the rest of the paper we might abuse the notation $\teth$ defined in \eqref{Eq::teteq} in the operator form 
 as $\teth(\cdot) := \gradh( \cdot) + \liftR(\jump{\cdot})+ \liftL(\avg{\cdot} - \uflux(\cdot))$. 
 
The only undefined part in the primal formulation~\eqref{eq:generalPrimal} is the explicit definition of the numerical fluxes.
We are going to consider four different formulations by adopting different $\uflux$ and $\sflux$:
\begin{enumerate}[label=(\roman*)]
 \item \textbf{BR1.} The BR1 formulation is defined as in \cite{bassi1997high1}, 
\begin{equation*}
\uflux = \begin{cases} \avg{\uh} &\text{on } \sklti  \\ 0 & \text{on } \skltb \end{cases},  \qquad \sflux =  \avg{ \sigh}\quad \text{on } \sklt.
\end{equation*}
Hence, $\teth = \gradh \uh + \liftR(\jump{\uh})$. 
\ifdetail
 and 
\begin{equation*}
- \sume \inted \inprods{\jump{\vh}}{\avg{\sigh}} \ds = \intdom \inprods{\liftR(\jump{\vh})}{\sigh} \dx = \intdom \inprods{\bm{a}(\uh, \teth)}{\liftR(\jump{\vh})} \dx.
\end{equation*}
\fi
This leads to the following quasi-linear formulation
\begin{align}
\label{Eq::Brone}
\difffrm (\uh, \vh ) = &\intdom \inprods{ \diff(\uh, \teth(\uh))}{\teth(\vh)} \dx.
\end{align}

 \item \textbf{BR2.} The BR2  formulation, inherited from the original definition  \cite{bassi1997high},  is defined as
\begin{equation*}
\uflux = \begin{cases} \avg{\uh} &\text{on } \sklti  \\
0 & \text{on } \skltb \end{cases}, \qquad 
\sflux  = \avg{ \galproj\big(\diff(\uh,  \nabla \uh  )  + \eta_\face \diff(\uh,\liftR^e(\jump{\uh}))\big)}  \quad \text{on } \sklt
\end{equation*} 
with some $\eta_\face>0$ as the stabilization parameter. 
\ifdetail
\mzdet{
This gives
\begin{align*}
- \sume \inted \inprods{\jump{\vh}}{\sflux} \ds &=  \intdom \inprods{\liftR(\jump{\vh})}{\diff(\uh, \nabla \uh)} \dx + \sume \eta_\face \intdom \inprods{\diff(\uh, \liftR^e(\jump{\uh}) ) }{ \liftR^e(\jump{\vh}) } \dx, 
\end{align*}
}
\fi
Here, $\teth$ is the same as in BR1, and   the primal formulation reads
\begin{align}
\label{Eq::Brtwo}
B(\uh, \vh)  &= \intdom \inprods{\diff(\uh, \teth)}{\gradh \vh} + \inprods{\diff(\uh, \gradh \uh)}{\liftR(\jump{\vh})} \dx  \\
& \quad + \sume \eta_\face \intdom \inprods{\diff(\uh, \liftR^e(\jump{\uh}) )}{ \liftR^e(\jump{\vh}) } \dx. \nonumber	
\end{align}
 
 \item \textbf{SIPG.}  Similar to \cite{arnold2002unified}, we choose the fluxes as
\begin{equation*}
\uflux = \begin{cases} \avg{\uh} & \text{on } \sklti \\ 0 & \text{on } \skltb \end{cases}, \qquad 
\sflux = \avg{ \galproj(\diff(\uh,  \nabla \uh  )} - \frac{\pen}{\he}  \jump{\uh} \quad \text{on } \sklt
\end{equation*}
with some penalty parameter $\pen > 0$. Similarly to BR2, the primal formulation is given as
\begin{align}
\label{Eq::sipg}
\difffrm(\uh, \vh )& = \intdom \inprods{ \diff(\uh, \teth)}{\gradh \vh} + \inprods{ \diff(\uh, \gradh \uh)}{\liftR(\jump{\vh})} \dx  + \sume  \frac{\pen}{\he}  \inted \inprods{\jump{\uh}}{\jump{\vh}} \ds.
\end{align}
 \item \textbf{LDG.} The LDG formulation, inherited from the original version in \cite{cockburn1998local}, can be obtained by setting 
\begin{equation*}
\uflux =\begin{cases} \avg{\uh} -  \inprods{‌\bldg}{ \jump{\uh}} & \text{on } \sklti \\  0 & \text{on } \skltb \end{cases}, \qquad  \sflux = \begin{cases} \avg{ \sigh}  + \bldg	\jump{\sigh} - \frac{\pen}{\he}  \jump{\uh} & \text{on } \sklti \\ \avg{ \sigh} - \frac{\pen}{\he}  \jump{\uh} & \text{on } \skltb \end{cases}
\end{equation*}
where $\bldg \in [\Ltwo(\sklti)]^2$ is some mesh-dependent parameter and constant on each edge.
Also like SIPG, we have the penalty parameter $\pen > 0$. Note that  in LDG formulation $\teth = \gradh \uh + \liftR(\jump{\uh }) +  \liftL(\inprods{\bldg}{\jump{\uh}})$.
\ifdetail
\mzdet{
So, 
\begin{align*}
- \inte \inprods{\jump{\vh}}{\sflux} \ds  &= \intdom \inprods{\diff(\uh, \teth(\uh))}{(\liftR(\jump{\vh}) + \liftL(\inprods{\bldg}{\jump{\vh}}))} \dx  +  \sume \frac{\pen}{\he}  \inted \inprods{\jump{\uh}}{\jump{\vh}} \ds.
\end{align*}
}\fi
 Finally, the primal formulation reads 
\begin{align}
\label{Eq::LDG}
\difffrm(\uh, \vh) = \intdom \inprods{\diff(\uh, \teth(\uh))}{ \teth(\vh)} \dx  + \sume \frac{\pen}{\he}  \inted \inprods{\jump{\uh}}{\jump{\vh}} \ds. 
\end{align}
\end{enumerate}

{\begin{remark}
Seeking comparison of  the current formulations with the already proposed versions in the literature, let us record a few observations. {Firstly, it is worth pointing out that those fluxes which depend only on $\uh$ and $\sigh$ (BR1, LDG) do not require any formal modification compared to the linear case.}
Furthermore, our LDG formulation is similar to  \cite{bustinza2004local} and the discretization proposed in \cite{gudi2008hpmixed} for a simpler form of nonlinearity in the diffusion.


For the SIPG formulation, looking at \eqref{Eq::sipg} and exploiting Taylor's expansion \eqref{Eq::Taylor}, we can write 
\begin{align*}
\difffrm(\uh, \vh )& = \intdom \inprods{ \diff(\uh,  \gradh \uh)}{\gradh \vh} \dx + \sume \inted \dfrac{\pen}{\he} \inprods{\jump{\uh}}{\jump{\vh}} \ds\\
& \quad  + \intdom \inprods{ \difft(\uh, \gradh \uh, \teth) \liftR(\jump{\uh})}{\gradh \vh} + \inprods{ \diff(\uh, \gradh \uh)}{\liftR(\jump{\vh})} \dx  \\
& = \intdom \inprods{ \diff(\uh,  \gradh \uh)}{\gradh \vh} \dx  + \sume \inted \dfrac{\pen}{\he} \inprods{\jump{\uh}}{\jump{\vh}} \ds \\
& \quad - \sume \inted   \inprods{\jump{\uh}}{\avg{\galproj(\difft(\uh, \nabla \uh, \teth)  \nabla \vh)}} + \inprods{\avg{ \galproj(\diff(\uh, \nabla \uh))}}{\jump{\vh}} \ds.
\end{align*}
Ignoring the boundary terms, the primal formulation proposed in  \cite{gudi2008hp} (or \cite{hartmann2005symmetric} in simpler case) reads as
\begin{align*}
\difffrm(\uh, \vh ) & = \intdom \inprods{ \diff(\uh,  \gradh \uh)}{\gradh \vh} \dx + \sume \inted \dfrac{\pen}{\he} \inprods{\jump{\uh}}{\jump{\vh}} \ds\\
& \quad - \sume \inted   \inprods{\jump{\uh}}{\avg{\diffz(\uh, \gradh \uh, \teth)  \nabla \vh}} + \inprods{\avg{ \diff(\uh, \gradh \uh)}}{\jump{\vh}} \ds
\end{align*}		
which shows two differences  in the second term: in \cite{gudi2008hp} there is  $\diffz$ instead of the average  $\difft$, which is a consequence of the direct consideration of the primal form. Also in our formulation we have an additional Galerkin projection, which seems essential to  later proofs in the paper.
Also one might notice that while $\diffz$ needs to be computed explicitly in \cite{gudi2008hp}, in our formulation  $\difft$ only appears in the analysis and there is no need to compute and implement this term.

Furthermore, comparing the current primal BR2 formulation with the version proposed in \cite{bassi2005discontinuous} for the case of nonlinear Helmholtz equation, and by following similar arguments to those  shown  for SIPG, one observes that the only difference of two formulations is in the application of the additional Galerkin projection in our formulation. 
\ifdetail\mrg{\mzdet{Add these details in the thesis}}  \fi
\end{remark}
}

\section{Consistency and adjoint consistency}
\label{Sec::Consst}
It is clear that due to the discrete nature of the lifting operator and the Galerkin projection, the primal formulation \eqref{Eq::primal} is inconsistent with the exact solution (for all four presented methods). This is also the case for the formulations presented in  \cite{bustinza2004local} and \cite{perugia2002hp}. Moreover, the scheme is inconsistent with the smooth solution of its corresponding adjoint problem constructed by the linearization in the neighborhood of the exact solution. 

In order to investigate these consistency errors, we consider a smooth solution to \eqref{Eq::origPDE}. In sections~\ref{Subsec::consis} and \ref{Subsec::adjconsis} we  prove that, despite of the primal and adjoint inconsistency of the formulations, and in case of sufficiently regular solutions, one  obtains asymptotic consistency in the mesh refining limit, for primal as well as for the adjoint problem.  

In the following of this paper, due to the required regularity in the adjoint analysis, we assume that the exact solution of the problem \eqref{Eq::origPDE}, denoted by $\uu$, belongs to $\Wmp{\infty}{2} (\dom)$.

\ifdetail\mzdet{
One might confer with \cite[Remark 3.4]{gudi2008hp} for the discussion on a weaker assumption on derivatives of $\diff$.
However, in the consistency analysis of the primal formulation in section~\ref{Subsec::consis} we only need $\uu \in \Hm{{1+\epsilon}}$, $\epsilon > 0$ (to guarantee $\jump{\uu} = 0$). 
On the other hand in section~\ref{Subsec::Ltwo} we need a technical assumption $\uu \in \Hm{{5/2}}(\dom, \triag)$ for obtaining the optimal error convergence in $\Ltwo$ norm. In order to avoid working with different regularity at each section we adopt the strongest one  and set $\uu \in \Wmp{\infty}{2}(\dom)$ (one might prefer an almost equivalent space, $\Hm{{3+\epsilon}}(\dom)$, $\epsilon > 0$). 
}\fi

\subsection{Consistency}
\label{Subsec::consis}
Due to the assumption~\ref{Ass::1}, the fact that $f \in \Ltwo(\dom)$ and $- \nabla \cdot \diff(\cdot, \uu, \nabla \uu) = f $ on $\dom$ for the exact solution $\uu$, we know that $\diff( \uu, \nabla \uu) \in H(\mathrm{div}; \dom)$. Hence, $\jump{\diff( \uu, \nabla \uu)} = 0$ on each $\face \in \sklti$.
Also note that  we have $\teth(\uu) = \nabla \uu$ for all schemes. Let us first  consider the  schemes BR1, BR2, and SIPG. For these schemes, and for any $\vv \in \Vh$, one can rewrite the left hand side of \eqref{eq:generalPrimal} as
\begin{align*}
\difffrm(\uu, \vv ) & =  \intdom \inprods{ \diff(\uu,  \nabla \uu)}{\gradh  \vv} \dx   + \inprods{ \diff(\uu, \nabla \uu)}{\liftR(\jump{\vv})} \dx \\
& = \intdom \inprods{ \diff(\uu, \nabla  \uu)}{\gradh  \vv} \dx  - \! \sume \! \inted \inprods{ \avg{\diff(\uu, \nabla \uu)}}{\jump{\vv}} \ds + \! \sume \! \inted \inprods{ \avg{(I - \galproj)(\diff(\uu, \nabla \uu))}}{\jump{\vv}} \ds. 
\end{align*} 
Applying the divergence theorem on the first two terms on the right hand side, and using the continuity of  $\diff( \uu, \nabla \uu) \cdot \no_\telem$  on the interfaces yield
\begin{align*}
\difffrm(\uu, \vv )& = - \intdom \nabla \cdot \diff(\uu,  \nabla  \uu) \vv \dx  +  \sume \inted \inprods{ \avg{(I - \galproj)(\diff(\uu, \nabla \uu))}}{\jump{\vv}} \ds.
\end{align*} 
Noting  \eqref{Eq::origPDE}, the consistency error for the primal formulation is the same for BR1, BR2 and SIPG; that is 
\begin{align}
\label{eq:consErrPrimal}
\prmcnst(\uu, \vv) := \difffrm(\uu, \vv ) - \rhsfrm(\vv) =  \sume \inted \inprods{ \avg{(I - \galproj)(\diff(\uu, \nabla \uu))}}{\jump{\vv}} \ds, \qquad \forall \vv \in \Vh.
\end{align}
Similarly, for LDG scheme, following the same lines as \cite{bustinza2004local} we obtain
\begin{equation}
\prmcnst(\uu, \vv) :=  \sume \inted \inprods{ \avg{(I - \galproj)(\diff(\uu, \nabla \uu))} }{\jump{\vv}} \ds - \sumei \inted \inprods{ \bldg \jump{(I - \galproj)(\diff(\uu, \nabla \uu))}  }{\jump{\vv}} \ds.
\end{equation}
In general, even for very regular solutions, or in case of linear diffusion $\diff(\uu, \nabla \uu) = \nabla \uu$, this consistency error is not zero and in fact is equal to the Galerkin projection error into the polynomial space $\Sigsp$. \ifdetail\mgg{\mzdet{But for $\diff(\uh, \nabla \uh) = \nabla \uh$  the Galerkin projection should be the identity.}}\fi  However, provided that the diffusion $\diff$  is regular enough, the formulations are asymptotically  consistent; i.e., $\Vert \prmcnst(\uu, \cdot) \Vert_{\Vhd} \to 0$ as $\h \to 0$. 
Note that $\Vert \cdot \Vert_{\Vhd}$ is the dual norm on space $\Vh$ defined as 
\begin{equation}
\Vert A(\cdot) \Vert_{\Vhd} := \sup_{0 \neq \ww \in \Vh} \dfrac{ \vert A(\ww) \vert}{\ennorm{\ww}},
\end{equation}
where $A: \Vh \to \Vhd$ is a linear continuous operator on $\Vh$.

In order to prove this asymptotic consistency, one might find  an upper bound for $\prmcnst(\uu,  \vv)$ which vanishes as $\h$ goes to zero.  Here we present a generalized form of \cite[Lemma 5.2]{bustinza2004local} for different discretizations, which provides us with such an estimate:
\begin{lemma}
\label{Lem::Cnsst}
Assume $\diff(\uu, \nabla \uu) \in \Hm{{\smth_*}}( \dom, \triag)$ with some	 non-negative integer $\smth_*$. Then, there exists $\Ccon > 0$, independent of $\h$ but dependent on $\poldeg$ and $\smth_*$, such that for the schemes presented in section~\ref{Sec::DG} the following holds
\begin{equation}
\vert \difffrm(\uu, \ww) - \rhsfrm(\ww) \vert \leq \Ccon \Big( \sumk \h_\telem^{2\mu_*}  \Vert \diff(\uu, \nabla \uu) \Vert^2_{\Hm{{\smth_*}}(\telem)} \Big)^{1/2} \ennorm{\ww},
\end{equation}
 for all $\ww \in \Vh$, where $\mu_* = \min(\smth_*, \poldeg+1)$.
\end{lemma}
\begin{proof}
The proof is  provided in \cite{bustinza2004local} for the LDG scheme, which in case $\bldg  = 0$  gives the desired result	 for BR1, BR2, and SIPG.
\end{proof}
\subsection{Adjoint consistency}
\label{Subsec::adjconsis}
Adjoint consistency is an important property in obtaining the optimal $\Ltwo$ convergence rate for the solution, and super-convergence of the target functionals (cf. \cite{hartmann2007adjoint} or \cite{pierce2000adjoint}). 
In order to apply Aubin-Nitsche duality argument, let us consider the following auxiliary dual problem
\begin{align}
\label{Eq::Adjoint}
- \nabla \cdot (\diffz(\uu, \nabla \uu) \nabla \psi ) + \diffu (\uu, \nabla \uu) \cdot \nabla \psi &= \uu -\uh,  &&\text{in } \dom, \\
\psi & = 0, &&\text{on } \partial \dom, 
\end{align}
where $\uu$ is the exact  solution of \eqref{Eq::origPDE}. From assumption~\ref{Ass::3} and $\uu \in \Wmp{\infty}{2} (\dom)$, and {provided that $\uh \in \Ltwo(\dom)$ (cf. section~\ref{Sec::stab})} one can check that
\begin{equation}
 \label{Eq::adjint_prop}
 \diffu(\uu, \nabla \uu) \in \Wmp{\infty}{1}(\dom), \quad  \diffz(\uu, \nabla \uu) \in \Wmp{\infty}{1}(\dom), \quad  \uu -\uh \in \Ltwo(\dom).
\end{equation}
 
Using \eqref{Eq::adjint_prop} combined with the convexity of $\dom$, it is a classical result that  the unique  solution of the adjoint problem, $\psi \in \Hm{2} \cap \Hm{1}_0 $ satisfies the following elliptic regularity  \cite[Theorem 9.1.22]{opac-b1081643}
\begin{equation}
\label{Eq::regularity}
\Vert \psi \Vert_{\Hm{2}(\dom)} \leq C \Vert \uu - \uh \Vert_{\Ltwo(\dom)}.
\end{equation}
Moreover,  from \eqref{Eq::adjint_prop} and the structure of  \eqref{Eq::Adjoint}, one can conclude that $\diffz(\uu, \nabla \uu) \nabla \psi \in H(\mathrm{div}, \dom)$. \ifdetail since $\uu - \uh \in \Ltwo(\dom)$,  and $\diffu (\uu, \nabla \uu) \cdot \nabla \psi \in \Hm{1}(\dom)$.\fi Therefore one gets  $\jump{\diffz(\uu, \nabla \uu) \nabla \psi} = 0$, on any interior edge $\face \in \sklti$.

In order to do the linearization, we take the Fr\'echet derivative of $\difffrm(\uu, \vv)$  around the exact solution $\uu$. For the BR1 formulation  one might get
\begin{align}
\label{Eq::Frechet_brone}
\difffrm'[\uu](\ww, \vv) = \intdom \inprods{\big( \diffu(\uu, \nabla  \uu) \ww + \diffz(\uu, \nabla  \uu) (\gradh  \ww + \liftR(\jump{\ww})) \big)}{\big(\gradh  \vv + \liftR(\jump{\vv})\big)} \dx.
\end{align}
Similarly, for BR2 and SIPG one has
\ifdetail
\mzdet{
with some lines of computation and using that decomposition in terms of BR1 as \eqref{Eq::decompos} one can write
\begin{align*}
\Tone(\uu,\vv) = -\intdom \inprods{\Big(\diff(\uu, \teth(\uu)) - \diff(\uu, \nabla \uu) \Big)}{\liftR (\jump{\vv})} \dx, 
\end{align*}
\begin{align*}
{\Tone}'[\uu](\ww,\vv) = -\intdom \inprods{\Big(\diffz(\uu, \nabla \uu)(\nabla \ww + \liftR(\jump{\ww})) - \diffz(\uu, \nabla \uu)(\nabla \ww )  \Big)}{\liftR (\jump{\vv})} \dx, 
\end{align*}
}\fi
\begin{align}
\label{Eq::Frechet_brtwo_sipg}
\difffrm'[\uu](\ww, \vv) &= \intdom \inprods{\big( \diffu(\uu, \nabla  \uu) \ww + \diffz(\uu, \nabla  \uu) (\gradh  \ww + \liftR(\jump{\ww})) \big)}{\big(\gradh  \vv + \liftR(\jump{\vv})\big)} \dx \nonumber \\
& 
\quad + \intdom \inprods{ \diffz(\uu, \nabla \uu) \liftR(\jump{\ww}) }{\liftR(\jump{\vv})} \dx  + F(\uu, \ww, \vv),
\end{align}
where the term $F$ is the Fr\'echet derivative of the penalty term; there holds
\begin{equation}
F(\uu, \ww, \vv) = \sume \inted \frac{\pen}{\he} \inprods{\jump{\ww}}{\jump{\vv}} \ds, \qquad  F(\uu, \ww, \vv) = \sume \eta_\face \intdom  \inprods{\liftR^\face(\jump{\ww})}{\liftR^\face(\jump{\vv})} \dx 
\end{equation}
for SIPG and BR2, respectively.
For the LDG formulation we have
\begin{align*}
\difffrm'[\uu](\ww, \vv) \! &= \! \! \intdom \! \! \inprods{\big( \diffu(\uu, \nabla \uu) \ww + \diffu(\uu, \nabla  \uu) (\nabla \ww + \liftR(\jump{\ww}) + \liftL(\bldg \cdot \jump{\ww})) \big)\!}{\!(\gradh  \vv + \liftR(\jump{\vv}) + \liftL(\bldg \cdot \jump{\vv})) } \dx \\
&\quad +  \sume \inted \frac{\pen}{\he} \inprods{\jump{\ww}}{\jump{\vv}} \ds.
\end{align*}
Now, let us consider $\difffrm'[\uu](\ww, \psi)$ for the smooth exact solution of the dual problem, $\psi \in \Hm{2}(\dom)$; one might write for BR1, BR2 and SIPG,
\begin{align*}
\difffrm'[\uu](\vv, \psi) &= \intdom \inprods{\big( \diffu(\uu, \nabla  \uu) \vv + \diffz(\uu, \nabla \uu) (\nabla \vv + \liftR(\jump{\vv})) \big)}{\gradh  \psi} \dx \\
&= \intdom \inprods{ \diffu(\uu, \nabla  \uu) \vv }{\gradh  \psi} + \inprods{( \diffz(\uu, \nabla \uu) \gradh  \psi)}{  \gradh  \vv } \dx  + \intdom \inprods{( \diffz(\uu, \nabla  \uu) \gradh  \psi)}{ \liftR(\jump{\vv})} \dx \\
& = \intdom {\big(\diffu(\uu, \nabla  \uu) \cdot \gradh  \psi  - \divh (\diffz(\uu, \nabla \uu) \gradh  \psi)\big)}{\vv} \dx  \\
& \quad  +  \sume \inted \inprods{\avg{(I - \galproj)(\diffz(\uu, \nabla \uu) \nabla \psi)}}{\jump{\vv}} \ds \\
& = \intdom {(\uu - \uh)}{\vv} \dx  +  \dualcnst(\uu, \vv, \psi)
\end{align*}
for all $\vv \in \Vh$, with the following definition of the consistency error for the dual problem \eqref{Eq::Adjoint}
\begin{equation}
\label{eq:consErrDual}
\dualcnst(\uu, \vv, \psi) := \sume \inted \inprods{\avg{(I - \galproj)(\diff_z(\uu, \nabla \uu) \nabla \psi)}}{\jump{\vv}} \ds. 
\end{equation}
Similarly, for the LDG formulation one can show
\begin{equation*}
\dualcnst(\uu, \vv, \psi) = \sume \inted  \inprods{ \avg{(I - \galproj)(\diff_z(\uu, \nabla \uu) \nabla \psi)} }{\jump{\vv}} \ds  -\sumei \inted  \inprods{ \bldg \jump{(I - \galproj)(\diff_z(\uu, \nabla \uu) \nabla \psi)}}{\jump{\vv}} \ds. 
\end{equation*}
Presence of a non-zero consistency error for the dual problem reveals the dual inconsistency of the proposed scheme, but as we are going to show that one can obtain the \emph{asymptotic dual consistency} of the scheme as proposed in \cite{lu2005posteriori}; the quasilinear form $\difffrm$ is called  asymptotic dual consistent with the target functional $\tarfrm: \Real \to \Real$ if the following holds
\begin{equation}
\lim_{\h \to 0}  \Vert \difffrm'[\uu](\cdot, \psi) - \tarfrm'[\uu](\cdot) \vert \Vert_{\Vhd} = 0,
\end{equation} 
where $\uu$ and $\psi$ are the exact solutions of the primal and dual problems,  \eqref{Eq::origPDE} and \eqref{Eq::Adjoint} respectively.
Here $\tarfrm'[\uu]$ is the Fr\'echet derivative of  $\tarfrm$, which in our analysis and according to \eqref{Eq::Adjoint} is, 
\begin{equation}
\label{Eq::target}
\tarfrm(\uu) = \sumk \intk \frac{1}{2}(\uu - \uh)^2 \dx, \qquad \tarfrm'[\uu](\ww) = \sumk \intk (\uu - \uh) \ww \dx, \qquad \forall \uu , \ww\in \Vh.
\end{equation} 
As we will prove later in  section~\ref{Subsec::Ltwo},  such an asymptotic adjoint consistency leads to the optimal convergence rate in $\Ltwo$ norm. This property has been already investigated e.g., in \cite{oliver2009analysis} by numerical tests. Let us present the following lemma for the adjoint consistency error
\begin{lemma}
\label{Lemm::Adjoint_Consst}
Assume that the exact solution of the adjoint problem \eqref{Eq::Adjoint} has elliptic regularity, i.e., \eqref{Eq::regularity} holds. Then, there exists $\Cacon > 0$, independent of $\h$ but dependent on  $\poldeg$, such that
\ifdetail\mrg{\mzdet{$\Wmp{\infty}{1}(\dom, \triag)$ same as $\Wmp{\infty}{1}(\dom)$}}\fi
\begin{equation}
\dualcnst(\uu, \vv, \psi)  \leq \Cacon h \Vert \diff_z( \uu, \nabla \uu) \Vert_{\Wmp{\infty}{1}(\dom)} \Vert \psi \Vert_{\Hm{2}(\dom)} \ennorm{\vv},
\end{equation}
for all $\vv \in \Vh$.
\end{lemma}
\ifdetail\mrg{\mzdet{In \cite{oliver2009analysis} the assumption is mentioned differently}}\fi
\begin{proof}
Using assumption~\ref{Ass::3}  and Lemmas~\ref{Lemma::lift_jump} and \ref{Lem::hp_approx}, for BR2 and SIPG we have
\begin{align*}
\vert \dualcnst (\uu, \vv, \psi)  \vert &=  \Big\vert \sume \inted \inprods{\avg{(I - \galproj)(\diffz(\uu, \nabla \uu) \nabla \psi)}}{\jump{\vv}} \ds    \Big\vert \\
& \leq  \Big( \sume  \Vert (I - \galproj)(\diffz(\uu, \nabla \uu) \nabla \psi) \Vert^2_{\Ltwo(\face)}    \Big)^{1/2}  \Big(\sume  \Vert \jump{\vv} \Vert^2_{\Ltwo(\face)}    \Big)^{1/2} \\
& \leq \Big( \CA \sumk  \hk \Vert (\diffz(\uu, \nabla \uu) \nabla \psi) \Vert^2_{\Hm{1}(\telem)}    \Big)^{1/2}  \Big( \Cr^{-1} \sume  \he \Vert \liftRe(\jump{\vv}) \Vert^2_{\Ltwo(\dom)}    \Big)^{1/2}  \\
& \leq C  \h \Vert (\diff_z(\uu, \nabla \uu) \nabla \psi) \Vert_{\Hm{1}(\dom, \triag)}   \liftnorm{\vv}.
\end{align*}

For LDG there exists additional term which can be bounded similarly as
\begin{align*}
\Big \vert \sumei \inted  \inprods{ \bldg \jump{(I - \galproj)(\diffz(\uu, \nabla \uu) \nabla \psi)}}{\jump{\vv}} \ds \Big \vert &\leq C \h \Vert \bldg \Vert_{[\Lp{\infty}(\sklti)]^2} \Vert \diff_z(\uu, \nabla \uu) \nabla \psi \Vert_{\Hm{1}(\dom, \triag)} \liftnorm{\vv}. 
\end{align*}
Using the boundedness of parameter $\bldg$ and the fact that
\begin{equation*}
\Vert \diffz(\uu, \nabla \uu) \nabla \psi\Vert_{\Hm{1}(\dom, \triag)} \leq \Vert \diffz(\uu, \nabla \uu)  \Vert_{\Wmp{\infty}{1}(\dom, \triag)} \Vert \psi \Vert_{\Hm{2}(\dom, \triag)} \leq C,
\end{equation*}
completes the proof of the lemma with sufficiently large $\Cacon$.
\end{proof}

Using the result of the Lemma~\ref{Lemm::Adjoint_Consst}, smoothness of $\diff_z$ and elliptic regularity of $\psi$ yields
\begin{equation}
\lim_{\h \to 0} \Vert \dualcnst(\uu, \cdot, \psi) \Vert_{\Vhd} = 0, 
\end{equation}
which shows the asymptotic adjoint consistency of the proposed formulations.
\section{Stability}
\label{Sec::stab}
In this section we prove the stability of the approximated solution of \eqref{Eq::primal}. Let us remark that in this section we do not assume neither the monotonicity nor the Lipschitz continuity of the diffusion operator. The only additional assumption we need  is mentioned as the following remark:  
\begin{remark}
Based on the examples in section~\ref{Sec::Int}, and our main interest to interpret these problems as nonlinear diffusion phenomenon, it looks natural to assume the following 
\begin{equation}
\label{Eq::nodiff}
\diff(x, \eta, 0) \equiv 0, \qquad \forall  x \in \Real^2, \eta \in \Real,
\end{equation}
which tells that there is no diffusive behavior in the absence of the gradient.

Using \eqref{Eq::nodiff} and the Taylor's expansion \eqref{Eq::Taylor},  as well as assumption~\ref{Ass::3} in  form of \eqref{Eq::ass_stab_bnd}, one may write
\begin{equation}
\label{Eq::PosDef}
\inprods{\diff(x, \eta, \bm{\xi})}{\bm{\xi}} = \inprods{\difft(x, \eta, \bm{\xi})}{\bm{\xi}} \geq \lambda \vert \bm{\xi} \vert^2, \qquad \forall  x \in \Real^2, \eta \in \Real, \bm{\xi}\in \Real^2.
\end{equation}
We will exploit this result in the rest of this section.
\end{remark}

It is well-known that the BR1 method is only weakly-stable (cf. \cite{arnold2002unified}), so we only present the stability result for BR2, SIPG and LDG  as the following lemma:
\begin{lemma}
\label{Lem::Stab}
Let us assume $\uh \in \Vsp$ is the solution of the primal formulation \eqref{Eq::primal}. Then for LDG, BR2 and SIPG methods the following stability result holds
\begin{equation}
\label{Eq::Stab}
\ennorm{\uh} \leq \dfrac{\CF}{\Cco} \Vert f \Vert_{\Ltwo(\dom)}.
\end{equation}
Here, $\CF$ is the continuity constant of linear form $\rhsfrm(\cdot)$, that is
\begin{equation}
\label{Eq::CF}
\vert \rhsfrm(\vv) \vert \leq \CF \Vert f \Vert_{\Ltwo(\dom)} \ennorm{\vv}, \qquad \forall \vv \in \Vh,
\end{equation}
while $\Cco > 0$ is the following coercivity constant of the operator $\difffrm$ as
\begin{equation}
\label{Eq::coercivity}
 \Cco \ennorm{\vh}^2  \leq \difffrm(\vh, \vh), \qquad \forall \vh \in \Vsp.
\end{equation}
\end{lemma}
\begin{proof}
For different formulations, we  insert the test function $\vh = \uh$ and investigate if \eqref{Eq::coercivity} actually holds.
Using the fact that $\difffrm(\uh, \uh) = \rhsfrm(\uh)$, the rest of the proof is complete by the following line
\begin{equation}
 \ennorm{\uh}^2 \leq \dfrac{1}{\Cco} \rhsfrm(\uh) \leq \dfrac{\CF}{\Cco}  \Vert f \Vert_{\Ltwo(\dom)}   \ennorm{\uh}.
\end{equation}
So the only missing step is the coercivity proof, for which we check $\difffrm(\uh, \uh)$ for different formulations:
\begin{enumerate}[label= ({\roman*})]
\item LDG: From \eqref{Eq::LDG} and \eqref{Eq::teteq} one has
\begin{align*}
\difffrm(\uh, \uh) & = \intdom \inprods{  \diff \big(\uh,  \teth \big)}{ \teth  } \dx +\sume \dfrac{\pen}{\he}\inted \inprods{\jump{\uh}}{\jump{\uh}} \ds.		
\end{align*}
From the property \eqref{Eq::nodiff} we have
\begin{align}
\intdom \inprods{  \diff \big(\uh,  \teth \big)}{ \teth  } \dx &\geq  \lambda \Vert   \gradh \uh + \liftR(\jump{\uh}) + \liftL(\inprods{\bldg}{\jump{\uh}}) \Vert_{\Ltwo(\dom)}^2 \nonumber \\
& \geq  \lambda \Big[(1-\delta) \Vert   \gradh \uh \Vert^2_{\Ltwo(\dom)} + (1-\frac{1}{\delta}) \Vert \liftR(\jump{\uh}) + \liftL(\inprods{\bldg}{ \jump{\uh}}) \Vert_{\Ltwo(\dom)}^2 \Big]
\end{align}
using Young's inequality with $\delta \in (0, 1)$. Now using \eqref{Eq::liftnorm_bound} and the boundedness of $\bldg$, one has
\begin{align}
\intdom \inprods{  \diff \big(\uh,  \teth \big)}{ \teth  } \dx & \geq  \lambda \Big[(1-\delta) \Vert   \gradh \uh \Vert^2_{\Ltwo(\dom)} + C (1-\frac{1}{\delta}) \liftnorm{\uh}^2 \Big],
\end{align}
for some constant $C$. Then \eqref{Eq::lift_jump} gives the following 
\begin{align}
\difffrm(\uh, \uh) \geq \lambda (1-\delta) \Vert   \gradh \uh \Vert^2_{\Ltwo(\dom)} +  \big( {\dfrac{\pen}{\CR^2}} + \lambda C (1-\frac{1}{\delta}) \big) \liftnorm{\uh}^2,
\end{align}
which indicates the condition $\pen > 0$ as the coercivity requirement of LDG formulation. This coincides with the result for the Poisson problems (cf. \cite{arnold2002unified}) and the version of LDG discussed in \cite{may2016spacetime}.

\item BR2, SIPG: We combine the proof for these two methods by writing them as a variation of BR1 method. This treatment will be used once more in section \ref{Subsec::strong_mono}. Using \eqref{Eq::Brone}, \eqref{Eq::Brtwo} and \eqref{Eq::sipg}, one may write the following decomposition of primal formulation; for all $\vv, \ww \in \Vh $
\begin{equation}
\label{Eq::decompos}
\difffrm_{\sipg, \brtwo}(\vv, \ww) = \difffrm_\brone(\vv, \ww) + \Tone(\vv, \ww) + \Ttwo_{\sipg, \brtwo}(\vv, \ww),
\end{equation}
where $\difffrm_\brone$ is the primal formulation of BR1 scheme and the next two terms are defined as
\begin{align}
\Tone(\vv,\ww) &= -\intdom \inprods{\Big(\diff(\vv, \teth(\vv)) - \diff(\vv, \gradh \vv) \Big)}{\liftR (\jump{\ww})} \dx, \label{Eq::T_one} \\
\Ttwo_\sipg(\ww, \vv) &= \sume \dfrac{\pen}{\he}  \inted  \inprods{\jump{\vv}}{\jump{ \ww}} \ds, \label{Eq::T_two_sipg}\\
\Ttwo_\brtwo(\ww, \vv) &= \sume \eta_\face \intdom  \inprods{\diff(\vv,\liftR^\face(\jump{\vv}))}{\liftR^\face(\jump{ \ww})} \dx \label{Eq::T_two_brtwo}.
\end{align}
Using \eqref{Eq::PosDef} one can easily show the positive semi-definiteness of $\difffrm_{\text{BR1}}(\uh, \uh)$ as \eqref{Eq::Brone}.
For $\Tone(\uh, \uh) $ using the Taylor's expansion \eqref{Eq::Taylor} we have
\begin{align*}
\Tone(\uh, \uh)   & = - \intdom \inprods{\difft(\uh, \gradh \uh, \teth)\liftR(\jump{\uh})}{\liftR(\jump{\uh})} \dx,
\end{align*}
and for the remaining term $\Ttwo_{\brtwo, \sipg}$ one can obtain
\begin{align*}
\Ttwo_{\sipg}(\uh, \uh) &= \sume \dfrac{\pen}{\he} \Vert   \jump{\uh} \Vert^2_{\Ltwo(\face)} \geq    \sume \dfrac{\pen}{{\CR^2}} \Vert  \liftR^\face(\jump{\uh}) \Vert^2_{\Ltwo(\dom)}, \\
\Ttwo_{\brtwo}(\uh, \uh) &\geq \lambda \eta_\face  \sume  \Vert \liftR^\face(\jump{\uh})) \Vert^2_{\Ltwo(\dom)}.
\end{align*}
Here we have used~\eqref{Eq::lift_jump} and~\eqref{Eq::PosDef} in the first  and second equation, respectively.
Henceforth, we adopt the  following notation for a constant $\CT > 0$
\begin{equation}
\label{Eq::CT}
\sipg: \CT := \min_{\face \in \sklt} \{  \pen/{\CR^2} \}, \qquad \brtwo: \CT := \min_{\face \in \sklt} \{ \lambda  \eta_\face\}.
\end{equation}
Hence, using \eqref{Eq::decompos}, \eqref{Eq::PosDef} and \eqref{Eq::CT}, one might write 
\begin{align*}
\difffrm_{\sipg, \brtwo}(\uh, \uh) &\geq \intdom \inprods{ \diff(\uh, \teth)}{\teth} \dx - \intdom \inprods{\difft(\uh,   \gradh \uh, \teth)\liftR(\jump{\uh})}{\liftR(\jump{\uh})} \dx \\
& \quad + \CT \sume  \Vert \liftR^\face(\jump{\uh})) \Vert^2_{\Ltwo(\dom)}  \\
&\geq \lambda \Vert \teth(\uh)) \Vert^2_{\Ltwo(\dom)} - \Lambda \Vert \liftR(\jump{\uh}) \Vert^2_{\Ltwo(\dom)}  + \CT  \liftnorm{\uh}^2.
\end{align*}
Applying Young inequality with $0 < \delta < 1$ and \eqref{Eq::liftcauchy}, we arrive at
\ifdetail 
\mzdet{
\begin{align*}
\theta^2 =  (\nabla + r)^2 = \nabla^2 + r^2 + 2 \nabla r \geq \nabla^2 + r^2 + 2 \nabla r 
\end{align*}
}\fi
\begin{align*}
\difffrm_{\sipg, \brtwo}(\uh, \uh) &\geq \lambda(1 - \delta) \Vert \gradh \uh \Vert^2_{\Ltwo(\dom)} - \non \big(\Lambda + \lambda(\dfrac{1}{\delta} - 1)\big) \liftnorm{\uh}^2 
+ \CT \liftnorm{\uh}^2 
\end{align*}
This leads to the following criterion
$
\non (\Lambda + \lambda(\dfrac{1}{\delta} - 1)) \leq \CT$, and since $\delta $ can be arbitrary chosen close to $1$, the margin for stability is $\CT > \non \Lambda $ which can be written separately for BR2 and SIPG as 
\ifdetail\mrg{\mzdet{I think we should have strict inequalities here. 
Yes, since we do not want $\delta$ goes to zero}}\fi
\begin{align}
\label{Eq::stab_br_sipg}
\sipg: \pen > {\CR^2} \Lambda \non, \qquad \brtwo:  \eta_\face >  \non \frac{\Lambda}{\lambda}.
\end{align}
This shows that by choosing sufficiently large penalty parameter, as \eqref{Eq::stab_br_sipg} suggests, one can ensure the stability of the solution of SIPG and BR2 formulations. \ifdetail\mgg{Maybe we should also highlight these explicit estimates in the introduction}\fi
\end{enumerate}
This completes the proof of coercivity as well the proof of the lemma.
\end{proof}
\begin{remark}
The result of Lemma~\ref{Lem::Stab} shows that the stability estimate is available in the degenerate case, i.e. where $\lambda = 0$, for SIPG method while for BR2 the uniform lower bound $\lambda$ should be uniformly larger than zero.
In the special case when $\diff(\uu, \nabla \uu) = a(\uu) \nabla \uu $, by the same lines of argument as in \cite{zakerzadeh2017entropy},  one can check that the stability criteria for SIPG remains unchanged while the BR2 one reduces to $\eta_\face \geq  \non$, which is consistent with the result of Poisson problem (cf. \cite{arnold2002unified}). 
\end{remark}
\section{Existence and uniqueness of the discrete solution}
\label{Sec::Exist}
In the following sections of this paper, we only consider the BR2 and SIPG formulations. The reason for doing this is the fact that our discussions on the LDG method will follow basically the arguments presented in \cite{bustinza2004local}. Moreover, the BR1 method is not so common due to non-compact stencil and instability~\cite{arnold2002unified}. 
\ifdetail
On the other hand, as we are going to analyze SIPG and BR2 by recasting them in terms of BR1 scheme augmented by additional terms according to \eqref{Eq::decompos},  one can easily reproduce the results for BR1 by ignoring the additional terms in the analysis. 
\fi

We first prove the strong monotonicity and Lipschitz continuity of the nonlinear operator corresponding to the primal formulation  in sections~\ref{Subsec::strong_mono} and \ref{Subsec::Lipsc}, respectively. Then we prove the uniqueness of the discrete solution and present a Strang type error estimate.
\subsection{Strong monotonicity}
\label{Subsec::strong_mono}
Let us start with the following lemma
\begin{lemma}
\label{Lem::Strong}
If the diffusion operator $\diff$ satisfies the strong monotonicity property as 
\eqref{Eq::Strong}, then there exists $\CSM > 0$ such that
\begin{equation}
\difffrm(\vv, \vv - \ww ) - \difffrm(\ww, \vv - \ww ) \geq \CSM \ennorm{ \vv - \ww}^2
\end{equation}
for all $\ww, \vv \in \Vh$.
\end{lemma}

Before stating the proof, let us remark the following estimate which later will be used in the analysis:
\begin{lemma}
\label{Lem::Temp_strong}
For all $\ww \in \Vh$, the following holds
\begin{equation}
\Vert \teth(\ww) \Vert_{\Ltwo(\dom)}^2 \geq  \dfrac{1}{2} \ennorm{ \ww}^2 -  \eta \liftnorm{\ww} ^2,
\end{equation}
for $\eta \geq \non+\frac{1}{2}$.
\end{lemma}
\begin{proof}
Using Young's inequality with $ 0 < \delta < 1$, \eqref{Eq::teteq} and \eqref{Eq::liftcauchy}, one can write
\begin{align*}
\Vert  \teth(\ww) \Vert_{\Ltwo(\dom)}^2 &= \Vert \gradh \ww \Vert_{\Ltwo(\dom)}^2 + \Vert \liftR(\jump{\ww}) \Vert_{\Ltwo(\dom)}^2 - \intdom 2 \inprods{\liftR(\jump{ \ww })}{ \gradh \ww} \dx \\
\ifdetail
& \geq (1 - \delta) \Vert \gradh \ww \Vert_{\Ltwo(\dom)}^2  + (1 - \dfrac{1}{\delta}) \Vert \liftR(\jump{\ww}) \Vert_{\Ltwo(\dom)}^2 \\
\fi
& \geq (1 - \delta) \Vert \gradh \ww \Vert_{\Ltwo(\dom)}^2 + \non (1 - \dfrac{1}{\delta})   \sume \Vert \liftRe(\jump{\ww}) \Vert_{\Ltwo(\dom)}^2.
\end{align*}
Hence, by setting $\delta = 1/2$ we arrive at
\begin{align*}
\Vert\teth(\ww) \Vert_{\Ltwo(\dom)}^2 &\geq \frac{1}{2} \Vert \gradh \ww \Vert_{\Ltwo(\telem)}^2   - {\non} \liftnorm{\ww}^2. 
\end{align*}
Using the definition of energy norm \eqref{Eq::energynorm}, one can add $\pm 1/2 \liftnorm{\ww}^2$ to the right hand side and the result is obtained.
\end{proof}

Now we are ready to present the proof of Lemma~\ref{Lem::Strong}:
\begin{proof}[of Lemma~\ref{Lem::Strong}]
Using the decomposition  \eqref{Eq::decompos}, and setting $\xih = \vh - \wh$ one gets
\begin{align}
 \difffrm(\vh, \xih ) -  \difffrm(\wh, \xih ) &= \big[ \difffrm_\brone(\vh, \xih ) -  \difffrm_\brone(\wh, \xih )\big] \nonumber \\ 
 & \quad + \big[ \Tone (\vh, \xih) - \Tone (\wh, \xih) \big]+ \big[ \Ttwo (\vh, \xih) - \Ttwo (\wh, \xih) \big] \nonumber \\
 & = \big[ \difffrm_\brone(\vh, \xih ) -  \difffrm_\brone(\wh, \xih )\big] + F_1 + F_2.
\end{align}
\ifdetail
Using the definition of $\Tone$ and $\Ttwo$ as \eqref{Eq::T_one}, \eqref{Eq::T_two_sipg} and \eqref{Eq::T_two_brtwo}, one may write
\begin{align*}
F_1 &= -\intdom \inprods{(\diff(\vh, \teth(\vh)) - \diff(\wh, \teth(\wh)) - \diff(\vh, \gradh \vh) + \diff(\wh, \gradh \wh))}{\liftR (\jump{\xih})} \dx
\end{align*}
\begin{align*}
\sipg: F_2 = \sume \inted \dfrac{\pen}{\he} \inprods{\jump{\xih}}{\jump{\xih}} \ds, \quad 
\brtwo: F_2 = \sume \eta_\face \intdom  \inprods{(\diff(\vh,\liftR^\face(\jump{\vh})) - \diff(\wh,\liftR^\face(\jump{\wh})))}{\liftR^\face(\jump{\xih})} \dx
\end{align*}
\fi 
Strong monotonicity \eqref{Eq::Strong} and \eqref{Eq::Brone} readily imply
\begin{equation}
 \difffrm_\brone(\vh, \xih ) -  \difffrm_\brone(\wh, \xih )  \geq  \Csm \Vert \teth(\vh) - \teth(\wh) \Vert^2_{\Ltwo(\dom)}.
\end{equation}
For the term $F_1$, using \eqref{Eq::Lip}, \eqref{Eq::T_one} and a Young inequality with $ \delta >0$, one has
\ifdetail 
\mz{
\begin{align*}
F_1 & = \intdom \inprods{\Big(\diff(\vh, \teth(\vh)) - \diff(\wh, \teth(\wh)) - \diff(\vh, \gradh \vh) + \diff(\wh, \gradh \wh) \Big)}{\liftR (\jump{\xih})} \dx \\
 & = \intdom \inprods{\Big(\diff(\vh, \teth(\vh)) - \diff(\wh, \teth(\wh))  \Big)}{\liftR (\jump{\xih})}  \dx+  \intdom \inprods{\Big( - \diff(\vh, \gradh \vh) + \diff(\wh, \nabla \wh) \Big)}{\liftR (\jump{\xih})} \dx\\
 & \leq \dfrac{1}{\delta} \Vert \diff(\vh, \teth(\vh)) - \diff(\wh, \teth(\wh)) \Vert^2_{\Ltwo(\dom)}	+  \dfrac{\delta}{4} \Vert \liftR (\jump{\xih}) \Vert^2_{\Ltwo(\dom)} + \dfrac{1}{\delta} \Vert \diff(\vh, \gradh \vh) - \diff(\wh, \nabla \wh) \Vert^2_{\Ltwo(\dom)} + \dfrac{\delta}{4} \Vert \liftR (\jump{\xih}) \Vert^2_{\Ltwo(\dom)}
\end{align*}
}
\fi
\begin{align*}
F_1 & \geq 	- \dfrac{1}{\delta} \Big( \Vert \diff(\vh, \teth(\vh)) - \diff(\wh, \teth(\wh)) \Vert^2_{\Ltwo(\dom)}	\Big) \\
& \quad - \dfrac{1}{\delta} \Big(   \Vert \diff(\vh, \gradh \vh) - \diff(\wh, \gradh \wh) \Vert^2_{\Ltwo(\dom)} \Big)	 - \dfrac{\delta}{2} \Vert \liftR (\jump{\xih}) \Vert^2_{\Ltwo(\dom)} \\
\ifdetail
 & \geq  - \dfrac{\Clc^2}{\delta} \Vert \xih  \Vert^2_{\Ltwo(\dom)}  - \dfrac{\Clc^2}{\delta} \Vert \teth(\xih)  \Vert^2_{\Ltwo(\dom)} -  \dfrac{\Clc^2}{\delta} \Vert \xih  \Vert^2_{\Ltwo(\dom)}  - \dfrac{\Clc^2}{\delta} \Vert \nabla \xih  \Vert^2_{\Ltwo(\dom)}	  - \dfrac{\delta}{2} \Vert \liftR (\jump{\xih}) \Vert^2_{\Ltwo(\dom)}	\\
 \fi
& \geq  - \dfrac{2 \Clc^2}{\delta} \Vert \xih  \Vert^2_{\Ltwo(\dom)}  - \dfrac{\Clc^2}{\delta} \Vert \teth(\xih)  \Vert^2_{\Ltwo(\dom)}   - \dfrac{\Clc^2}{\delta} \Vert \gradh \xih  \Vert^2_{\Ltwo(\dom)}	  - \dfrac{\delta}{2} \Vert \liftR (\jump{\xih}) \Vert^2_{\Ltwo(\dom)}
\end{align*}
\ifdetail \mrg{\mzdet{$(a-b)^2 \leq 2a^2 + 2b^2$}} \fi
Using the fact that  $ - \vert \gradh \xih  \vert^2 \geq -  2 \vert \teth(\xih)  \vert^2   - 2  \vert \liftR(\jump{\xih})  \vert^2$ and \eqref{Eq::liftcauchy}, we arrive at
\begin{align*}
F_1 & \geq  - \dfrac{2 \Clc^2}{\delta} \Vert \xih  \Vert^2_{\Ltwo(\dom)}  - \dfrac{3\Clc^2}{\delta} \Vert \teth(\xih)  \Vert^2_{\Ltwo(\dom)} 	- \non \Big(\dfrac{2 \Clc^2}{\delta}+\dfrac{\delta}{2}  \Big) \sume  \Vert \liftRe (\jump{\xih}) \Vert^2_{\Ltwo(\dom)}.
\end{align*}
For the remaining term $F_2$, using the definitions \eqref{Eq::T_two_sipg} and \eqref{Eq::T_two_brtwo}, as well \eqref{Eq::lift_jump} and \eqref{Eq::Strong}, one has
\begin{align*}
F_2 \geq  \CR^{-2} \sume  \pen \Vert \liftR^\face(\jump{\xih}) \Vert^2_{\Ltwo(\dom)}, \qquad 
F_2 \geq   \Csm \sume  \eta_\face \Vert \liftR^\face(\jump{\xih}) \Vert^2_{\Ltwo(\dom)}
\end{align*}
for SIPG and BR2 methods, respectively. Similar to \eqref{Eq::CT}, we define $\CT'$ such that, for both methods
\begin{equation}
F_2(\wh, \vh) \geq  \CT' \sume  \Vert \liftR^\face(\jump{\xih}) \Vert^2_{\Ltwo(\dom)}.
\end{equation}
Hence, combining all terms one can write
\begin{align*}
 \difffrm(\vh, \xih ) - \difffrm(\wh, \xih )  & \geq \Csm \Vert \teth(\xih) \Vert^2_{\Ltwo(\dom)} + \CT' \liftnorm{\xih}^2 \\
 & \quad - \dfrac{2 \Clc^2}{\delta} \Vert \xih  \Vert^2_{\Ltwo(\dom)}  - \dfrac{3\Clc^2}{\delta} \Vert \teth(\xih)  \Vert^2_{\Ltwo(\dom)} 	- \non \Big(\dfrac{2 \Clc^2}{\delta}+\dfrac{\delta}{2}  \Big) \liftnorm{\xih}^2 
\end{align*}
Let us set $\eta$ to be a constant larger than $\non + 1/2$; then using Lemma~\ref{Lem::Temp_strong} and \eqref{Eq::energynrm_eq}  give
\begin{align*}
 \difffrm(\vh, \xih ) - \difffrm(\wh, \xih )  & \geq \dfrac{1}{2}\Big[\Csm - \dfrac{3\Clc^2}{\delta} -   \dfrac{4 \Clc^2 \Cen}{\delta}\Big] \ennorm{\xih}^2 \\
 & \quad + \Big[ \CT' -\big(\Csm - \dfrac{3\Clc^2}{\delta}\big) \eta - \non \big(\dfrac{2 \Clc^2}{\delta}+\dfrac{\delta}{2}  \big) \Big] \liftnorm{\xih}^2. 
\end{align*}
Now  consider arbitrary $ 0 < \CSM  < \frac{\Csm}{2}$ and set $\delta$ large enough such that the following holds
\begin{align*}
\Csm - \dfrac{ \Clc^2 (3 + 4 \Cen)}{\delta} > 2 \CSM, \quad \text{or} \quad 
\delta > \dfrac{\Csm - 2 \CSM}{\Clc^2 (3 + 4 \Cen)}.
\end{align*}
The only remaining free parameter is the stabilization parameter $\CT'$ and one can choose it sufficiently large, such that
\begin{align*}
\CT'  \geq (\Csm - \dfrac{3\Clc^2}{\delta}) \eta + \non \big(\dfrac{2 \Clc^2}{\delta}+\dfrac{\delta}{2} \big).
\end{align*}
Finally we arrive at
$ \difffrm(\vh, \xih ) - \difffrm(\wh, \xih )  \geq \CSM \ennorm{\xih}^2$, 
for both BR2 and SIPG and the proof completes.
\end{proof}

\subsection{Lipschitz continuity}
\label{Subsec::Lipsc}
For the Lipschitz continuity we have the following lemma
\begin{lemma}
\label{Lem::Glob_Lip}
If the diffusion operator $\diff$ satisfies the Lipschitz continuity property as 
\eqref{Eq::Strong}, there exists $\CLC < \infty$ independent of the mesh size such that
\begin{equation}
\vert \difffrm(\zz, \ww) - \difffrm(\vv, \ww)   \vert \leq \CLC  \ennorm{\zz - \vv } \ennorm{\ww},
\end{equation}
for all $\zz, \vv, \ww \in \Vh$. 
\end{lemma}
\begin{proof}
Using the decomposition  \eqref{Eq::decompos} and similar to Lemma~\ref{Lem::Strong} one gets
\begin{align}
 \difffrm(\zz, \ww ) -  \difffrm(\vv, \ww ) = \big[ \difffrm_\brone(\zz, \ww ) -  \difffrm_\brone(\vv, \ww )\big] + F_1 + F_2.
\end{align}
where 
\begin{align}
F_1 = \Tone (\zz, \ww) - \Tone (\vv, \ww), \qquad F_2 = \Ttwo (\zz, \ww) - \Ttwo (\vv, \ww).
\end{align}
It is  straightforward to show the Lipschitz continuity of $ \difffrm_\brone(\vv, \ww)$ using \eqref{Eq::Lip}. For $F_1$, noting the fact that  one may write
\begin{align*}
\vert \diff(\zz, \gradh \zz) - \diff(\vv, \gradh \vv)  - \diff(\zz, \teth(\zz)) + \diff(\vv, \teth(\vv)) \vert & \leq \Clc \Big( \vert \zz - \vv \vert^2 + \vert \teth(\zz) - \teth(\vv) \vert^2  \Big)^{1/2} \\
& \quad +  \Clc \Big( \vert \zz - \vv \vert^2 + \vert \gradh \zz - \gradh \vv \vert^2  \Big)^{1/2} \\
& \leq  C \Big( \vert \zz - \vv \vert + \vert \gradh (\zz - \vv) \vert + \liftR (\jump{\zz - \vv}) \vert \Big),
\end{align*}
and readily
\begin{align*}
F_1 &\leq C \intdom  \Big( \vert \zz - \vv \vert + \vert \gradh (\zz - \vv) \vert + \liftR (\jump{\zz - \vv}) \vert \Big) \vert \liftR (\jump{\ww}) \vert \dx  \leq \CLC \ennorm{\zz - \vv} \ennorm{\ww}.
\end{align*}
For BR2 scheme one may note, for any $\face \in \sklt$
\ifdetail
\mrg{for $a, b \geq 0$, $a^2 + b^2 \leq a^2 + b^2 + 2ab$ so $(a^2 + b^2)^{1/2} \leq a + b$}
\fi
\begin{align*}
\vert \diff(\zz,\liftRe(\jump{\zz}))- \diff(\vv,\liftRe(\jump{\vv})) \vert \leq \Clc \Big( \vert \zz - \vv \vert^2 + \vert \liftRe(\jump{\zz - \vv}) \vert^2 \Big)^{1/2} \leq \Clc \Big( \vert \zz - \vv \vert + \vert \liftRe(\jump{\zz - \vv}) \vert \Big) .
\end{align*}
On the other hand, using the fact the $\liftRe$ vanishes outside two neighbor elements,  
\ifdetail \mgg{\mzdet{why do we need that $\liftRe$ vanishes outside the two elements adjacent to $e$? otherwise I need to rewrite the sum $\sume$ to just include the neighbours to keep from having infinity integral over the domain} }
\fi
for the corresponding $F_2$ term we have,
\begin{align*}
F_2 &\leq  \sume \eta_\face \intdom \Clc \Big( \vert \zz - \vv \vert + \vert \liftRe(\jump{\zz - \vv}) \vert \Big)  \vert \liftRe(\jump{ \ww})\vert \dx  \leq \CLC \ennorm{\zz - \vv} \ennorm{\ww}.
\end{align*}
Finally, for SIPG method one can easily write
\begin{align*}
F_2 =   \sume \dfrac{\pen}{\he} \inted \inprods{\jump{\vv - \zz}}{ \jump{\ww}} \ds \leq \CLC \ennorm{\zz - \vv} \ennorm{\ww}.
\end{align*}
Combining the results for $\difffrm_\brone$, $F_1$ and $F_2$ concludes the proof. 
\end{proof}

Now we present the existence, uniqueness and stability result for the approximated solution $\uh$ as well as a  Strang type error estimate. Though we discussed the stability of the solution in section~\ref{Sec::stab}, note that the result of Lemma~\ref{Lem::exist} requires stronger condition on the diffusion operator (strong monotonicity and global Lipschitz continuity) than the more general result already discussed in section~\ref{Sec::stab}.
\begin{lemma}
\label{Lem::exist}
There exists a unique $\uh \in \Vsp$ solution of \eqref{Eq::primal}, which satisfies
\begin{equation}
\label{Eq::Strong_Stab}
\ennorm{\uh} \leq \dfrac{1}{\CSM}\Big[ \CF \Vert f \Vert_{\Ltwo(\dom)} + \Vert \difffrm(0, \cdot) \Vert_{\Vhd} \big].
\end{equation}
Moreover, the following error estimate holds
\begin{equation}
\ennorm{\uu - \uh} \leq \big( 1 + \dfrac{\CLC}{\CSM} \big) \inf_{\vh \in  \Vsp} \ennorm{\uu - \vh} + \dfrac{1}{\CSM} \sup_{0 \neq \wh \in \Vsp} \dfrac{\vert  \difffrm(\uu, \wh) - \rhsfrm(\wh)  \vert}{\ennorm{\wh}}, 
\end{equation}
where $\uu$ is the exact solution of \eqref{Eq::origPDE}.
\end{lemma}
\begin{proof}
Using the strong monotonicity and Lipschitz continuity (Lemmas~\ref{Lem::Strong} and \ref{Lem::Glob_Lip}) the unique solvability of \eqref{Eq::primal} can be proved by a well-known result as \cite[Theorem 3.2.23]{nevcas1983introduction} or \cite[Theorem 35.4]{vzenivsek1990nonlinear}, also see \cite{houston2005discontinuous}. For the stability proof we refer to \cite[Theorem 4.5]{bustinza2004local} and here we only present the proof of Strang type error estimate due to its application in the rest of our analysis.

Consider the error $ \er = \uu - \uh $ and decompose it as $\er  = \prer + \dser$, where $\prer = \uu - \vh$ and $\dser = \vh - \uh$ where $\vh \in \Vsp$. From Lemma~\ref{Lem::Strong} one has
\begin{align*}
\CSM \ennorm{\uh - \vh}^2 &\leq  \difffrm(\uh, \dser) -   \difffrm( \vh, \dser)  = \big[ \difffrm(\uh, \dser) -  \difffrm(\uu, \dser) \big] +  \big[ \difffrm(\uu, \dser) -  \difffrm( \vh, \dser) \big]
\end{align*}
While the first group of terms is the consistency error, applying Lemma~\ref{Lem::Glob_Lip} gives
\begin{equation}
\label{Eq::Energyerr_tmp}
\ennorm{\uh - \vh} \leq \dfrac{\CLC}{\CSM} \ennorm{\uu - \vh} + \dfrac{1}{\CSM} \sup_{0 \neq \wh \in \Vsp} \dfrac{\vert  \difffrm(\uu, \wh) - \rhsfrm(\wh)  \vert}{\ennorm{\wh}} 
\end{equation}
Applying a triangle inequality completes the proof.
\end{proof}

\section{A priori error estimates}
\label{Sec::Error}
In this section, we provide the error estimate of the BR2 and SIPG methods illustrated in section~\ref{Sec::DG}. 
In section~\ref{Subsec::Energ_error}, we prove the optimal error estimate in the energy norm $\ennorm{\cdot}$ using the Strang type error estimate in Lemma~\ref{Lem::exist} and the asymptotic consistency result in section~\ref{Subsec::consis}. In section~\ref{Subsec::Ltwo}, we prove the optimal error estimate in $\Ltwo$ norm exploiting the result on asymptotic adjoint consistency already provided in section~\ref{Subsec::adjconsis}.

\subsection{Energy norm error estimate}
\label{Subsec::Energ_error}
Combining the result of Lemmas~\ref{Lem::exist} and \ref{Lem::Cnsst}, gives the following corollary for the error estimate in the energy norm:
\begin{cor}
\label{Cor::Energy_error}
Let     $\uh$ and $\uu$ be the solution of \eqref{Eq::primal} and the exact solution of \eqref{Eq::origPDE}, respectively. Also assume that $\uu \in \Hm{\smth}(\dom, \triag) $ and $\diff(\uu, \nabla \uu) \in \Hm{{\smth_*}}(\dom, \triag) $. Then the following holds
\begin{equation}
\ennorm{\uu - \uh}^2 \leq \Cerr \Big( \sumk \hk^{2(\mu - 1)}  \Vert \uu \Vert^2_{\Hm{\smth}(\telem)}  + \sumk \hk^{2\mu_*}  \Vert \diff( \uu, \nabla \uu) \Vert^2_{\Hm{{\smth_*}}(\telem)}  \Big)
\end{equation}
with some $\Cerr >0$, $\mu = \min(\smth, \poldeg+1)$ and $\mu_* = \min(\smth_*, \poldeg+1)$. 
\end{cor}
\begin{proof} Using Lemmas~\ref{Lem::exist} and \ref{Lem::Cnsst} one can easily write
\begin{equation}
\ennorm{\uu - \uh} \leq \big( 1 + \dfrac{\CLC}{\CSM} \big) \inf_{\vh \in  \Vsp} \ennorm{\uu - \vh} + \dfrac{\Ccon}{\CSM} \Big( \sumk \hk^{2{\mu_*}}  \Vert \diff( \uu, \nabla \uu) \Vert^2_{\Hm{{\smth_*}}(\telem)} \Big)^{1/2}.  
\end{equation}
Choosing $\vh = \projh \uu$ and applying the approximation result in Lemma~\ref{Lem::Approx} completes the proof with choosing a sufficiently large $\Cerr$.
\end{proof}

Using the error decomposition $\er = \prer + \dser$ (as in Lemma~\ref{Lem::exist}), we also are interested in obtaining an estimate for $\dser = \projh \uu - \uh$. By setting $\vh = \projh \uu$ in \eqref{Eq::Energyerr_tmp} and similar to the proof of Corollary~\ref{Cor::Energy_error} one might get
 \begin{equation}
\ennorm{\uh - \projh \uu} \leq \dfrac{\CLC}{\CSM} \ennorm{\uu - \projh \uu} + \dfrac{\Ccon}{\CSM} \Big( \sumk \hk^{2\mu_*}  \Vert \diff(\uu, \nabla \uu) \Vert^2_{\Hm{{\smth_*}}(\telem)}  \Big)^{1/2},
 \end{equation}
which yields, for some $\Cerrt > 0$
 \begin{equation}
 \label{Eq::Energy_dserr}
\ennorm{\uh - \projh \uu} \leq \Cerrt \Big( \sumk \hk^{2(\mu - 1)}  \Vert \uu \Vert^2_{\Hm{\smth}(\telem)}  + \sumk \hk^{2\mu_*}  \Vert \diff( \uu, \nabla \uu) \Vert^2_{\Hm{{\smth_*}}(\telem)}  \Big)^{1/2}.
 \end{equation}

Moreover, using the definition of $\sigh$ and $\teth$ as \eqref{Eq::sigeq} and \eqref{Eq::teteq}, one can prove the corresponding lemma for their error estimate 
\begin{lemma}[Theorem 5.5 in \cite{bustinza2004local}] 
Under  the same assumptions as those of Corollary~\ref{Cor::Energy_error}, there exists $\Cerrt > 0$ independent of the mesh size, such that
\begin{equation}
 \Vert \tet - \teth \Vert_{[\Ltwo(\dom)]^2} \leq \Cerrt \Big( \sumk \hk^{2(\mu - 1)}  \Vert \uu \Vert^2_{\Hm{\smth}(\telem)}  + \sumk \hk^{2\mu_*}  \Vert \diff( \uu, \nabla \uu) \Vert^2_{\Hm{{\smth_*}}(\telem)}  \Big)^{1/2},
\end{equation}
and, not necessarily with the same $\Cerrt$, 
\begin{equation}
 \Vert \sig - \sigh \Vert_{[\Ltwo(\dom)]^2} \leq \Cerrt \Big( \sumk \hk^{2(\mu - 1)}  \Vert \uu \Vert^2_{\Hm{\smth}(\telem)}  + \sumk \hk^{2\mu_*}  \Vert \diff( \uu, \nabla \uu) \Vert^2_{\Hm{{\smth_*}}(\telem)}  \Big)^{1/2},
\end{equation}
where  $\mu = \min(\smth, \poldeg+1)$ and $\mu_* = \min(\smth_*, \poldeg+1)$. 
\end{lemma}

The proof follows the same lines as \cite{bustinza2004local} and we skip it here.
\subsection{$L_2$ norm error estimate}
\label{Subsec::Ltwo}
In this section, we  present the error estimate of the solution $\uh$ in the $\Ltwo$ norm. Recalling the adjoint problem \eqref{Eq::Adjoint}, and by inserting $\ww = \uu - \uh$ as the (infinite dimensional) test function in \eqref{Eq::Frechet_brtwo_sipg} and \eqref{Eq::target}, to obtain
\begin{align}
\label{Eq::ltwo_err_tmp}
	\ltwonrm{\uh-\uu}{\dom}^2 &=  \difffrm'[ \uu] (\uh-\uu,\psi) - \Big(B'[\uu](\uh-\uu, \psi) - J'[\uu](\uh-\uu) \Big) \nonumber \\
			& = \difffrm(\uh,\psi) - \difffrm(\uu,\psi) - \mathcal{N}(\uu, \uh, \psi) \pm \big( \difffrm(\uu, \psi_h) - \difffrm(\uh,\psi_h) \big) + \dualcnst(\uu, \er, \psi) \nonumber  \\
			\ifdetail
			& =\difffrm(\uh,\psi - \psi_h) - \difffrm(\uu,\psi - \psi_h)  - \big( \difffrm(\uu, \psi_h) - \difffrm(\uh,\psi_h) \big)  + \dualcnst(\uu, \uh, \psi) - \mathcal{N}(\uu, \er, \psi) \nonumber  \\
\fi
			& = \difffrm(\uh,\psi - \psi_h) - \difffrm(\uu,\psi - \psi_h) - \prmcnst(\uu,  \psi_h)  + \dualcnst(\uu, \er, \psi) - \mathcal{N}(\uu, \uh, \psi) 
\end{align}
where $\prmcnst$ and $\dualcnst$ are defined in~\eqref{eq:consErrPrimal} and~\eqref{eq:consErrDual} as the consistency errors of the primal and the dual problem, respectively, while $\mathcal{N}(\uu, \uh, \psi)$ is the second order linearization error. {Here, $\psi_h = \projh \psi \in \Vsp$}. By the Taylor's formula in section~\ref{Secsec::approx} we have (note that $\teth(\er) = \gradh \er +  \liftR(\jump{\er})$ as already defined)
\begin{align}
\label{Eq::ltwo_taylor}
\mathcal{N}(\uu, \uh, \psi) &= \sumk \intk R_a(\uu - \uh, \nabla \uu - \nabla \uh) \cdot \nabla \psi \dx \\
& = \intdom \inprods{\Big[ \tilde{\diff}_{\uu\uu}(\uu, \nabla \uu)\er^2 + \teth(\er)\tr \tilde{\diff}_{\bm{z} \bm{z}}(\uu, \nabla \uu) \teth(\er)	 + 2\inprods{\tilde{\diff}_{\uu \bm{z}}(\uu, \nabla \uu)}{\teth(\er)} \er  \Big]}{\gradh \psi} \dx. \nonumber 
\end{align}
\ifdetail
\mzdet{
Note that in this special case, since $\jump{\psi} = 0$ all the additional terms to BR1 vanish and we can simply do the analyze for the BR1 
\begin{equation}
\mathcal{N}(\uu, \uh, \psi) = \difffrm(\uh,\psi) - \difffrm(\uu,\psi) - \difffrm'[ \uu] (\uh-\uu,\psi)  
\end{equation}
\begin{align*}
\difffrm(\uh,\psi) - \difffrm(\uu,\psi) &=  \intdom \inprods{\Big[ \diffu(\uu, \nabla \uu)\er + \diffz(\uu, \nabla \uu)(\nabla \er + \liftR(\jump{\er}))  \Big]}{\gradh \psi} \dx \\
& \quad + \intdom \inprods{\Big[ \tilde{\diff}_{\uu, \uu}(\uu, \nabla \uu)\er^2 + \tilde{\diff}_{\bm{z} \bm{z}}(\uu, \nabla \uu)(\nabla \er + \liftR(\jump{\er})^2 + 2\tilde{\diff}_{\uu \bm{z}}(\uu, \nabla \uu)\er (\nabla \er + \liftR(\jump{\er}))  \Big]}{\gradh \psi} \dx 
\end{align*}
so
\begin{equation*}
\mathcal{N}(\uu, \uh, \psi) =\intdom \inprods{\Big[ \tilde{\diff}_{\uu, \uu}(\uu, \nabla \uu)\er^2 + \tilde{\diff}_{\bm{z} \bm{z}}(\uu, \nabla \uu)(\nabla \er + \liftR(\jump{\er})^2 + 2\tilde{\diff}_{\uu \bm{z}}(\uu, \nabla \uu)\er (\nabla \er + \liftR(\jump{\er}))  \Big]}{\gradh \psi} \dx 
\end{equation*}
}
\mrg{ \mzdet{in \cite{gudi2008hp} $\psi_h$ appears instead of $\psi$, I don't know the reason of difference but similar to the argument for (4.73) in \cite{gudi2008hp} we can handle that case.}}
\fi
First let us handle the error estimate of the last three terms  on the right hand side of \eqref{Eq::ltwo_err_tmp}, $\mathcal{N}(\uu, \uh, \psi)$, $\prmcnst(\uu, \psi_h)$ and $\dualcnst(\uu, \uh, \psi)$, in the following steps:
\begin{enumerate}[label= ({\roman*})]
\item Using \eqref{Eq::ltwo_taylor} and very similar arguments as \cite[Lemma 3.10]{gudi2008hp} we have
\begin{align*}
\vert \mathcal{N}(\uu, \uh, \psi) \vert & \leq C \Big[ \Vert \er \Vert_{\Lp{4}(\dom)} + \Vert \teth(\er)	 \Vert_{\Lp{4}(\dom)} \Big] \Big[ \Vert \er \Vert_{\Lp{2}(\dom)} + \Vert \teth(\er) \Vert_{\Lp{2}(\dom)} \Big] \Vert \psi \Vert_{\Wmp{4}{1}(\dom, \triag)} \\
&\leq C \Big[ \Vert \er \Vert_{\Wmp{4}{1}(\dom, \triag)} + \Vert \liftR(\jump{\er}) \Vert_{\Lp{4}(\dom)} \Big] \Big[ \Vert \er \Vert_{\Wmp{2}{1}(\dom, \triag)} + \Vert  \liftR(\jump{\er}) \Vert_{\Lp{2}(\dom)} \Big] \Vert \psi \Vert_{\Wmp{4}{1}(\dom, \triag)}.
\end{align*}
From the embedding theorem, the last term is bounded since we know $\Hm{2}(\dom) \subset  \Wmp{4}{1}(\telem)$, i.e., $\Vert \psi \Vert_{\Wmp{4}{1}(\dom, \triag)} \leq C \Vert \psi \Vert_{\Hm{2}(\dom, \triag)}$ by a uniform constant $C$.
\ifdetail\mzdet{
Look at the Dahmen's digram for Sobolev's embeddings. (for $d \leq 4$) I don't remember where I saw this result but it seems correct, we have $\Lp{q} \subset \Wmp{p}{1}$ for $q \leq q^*$ where $\frac{1}{q^*} = \frac{1}{n} - \frac{1}{p}$. Then for $\uu \in \Lp{q}$ we have
\begin{equation}
\Vert \uu \Vert_{\Wmp{q}{1}}^2 = \Vert \uu \Vert_{\Lp{q}}^2 + \Vert \nabla \uu \Vert_{\Lp{q}}^2   \leq C_1 \Vert \uu \Vert_{\Wmp{p}{1}}^2 + C_2 \Vert \nabla \uu \Vert_{\Wmp{p}{1}}^2 \leq  C_1 \Vert \uu \Vert_{\Wmp{p}{2}}^2 + C_2 \Vert \uu \Vert_{\Wmp{p}{2}}^2  
\end{equation} 
In our case $n=2, p = 2$ so $q^* = \infty$ (this does not mean the $\Linf$ boundedness).
}

\mrg{\mzdet{I don't think we need $\Wmp{2}{1}(\dom)$ and $\Wmp{2}{1}(\dom, \triag)$ is sufficient.}}
\fi
Using \eqref{Eq::energynrm_eq_H1} gives, since $\uu - \uh \in \Vh$
\begin{equation}
 \Vert \er \Vert_{\Wmp{2}{1}(\dom, \triag)} + \Vert  \liftR(\jump{\er}) \Vert_{\Lp{2}(\dom)}  \leq C \ennorm{\uu - \uh}.
\end{equation}
Now, the only term to handle (to obtain an additional order of $\h$) is $\Vert \uu - \uh \Vert_{\Wmp{4}{1}(\dom, \triag)}$ and $\Vert  \liftR(\jump{\er}) \Vert_{\Lp{4}(\dom)}$. By writing $\uu - \uh = \uu - \projh \uu + \projh \uu - \uh = \eta + \xi$,  one gets
\begin{align*}
\Vert \uu - \uh \Vert_{\Wmp{4}{1}(\dom, \triag)} + \Vert  \liftR(\jump{\er}) \Vert_{\Lp{4}(\dom)}&\leq \Vert \prer \Vert_{\Wmp{4}{1}(\dom, \triag)}  + \Vert \dser \Vert_{\Wmp{4}{1}(\dom, \triag)} + \Vert  \liftR(\jump{\er}) \Vert_{\Lp{4}(\dom)} .
\end{align*}
In case of $\poldeg \geq 2$ and $\uu \in \Hm{{5/2}}(\dom)$ (note that $ \poldeg + 1 > 5/2$),  application of Lemma~\ref{Lem::hp_approx} gives
\begin{align*}
\Vert \prer \Vert_{\Wmp{4}{1}(\dom, \triag)} &\leq  C \Big( \sumk \hk^4 \Vert \uu \Vert^4_{\Hm{{5/2}}(\telem)} \Big)^{1/4} \leq  C \h \Big( \sumk \Vert \uu \Vert^2_{\Hm{{5/2}}(\telem)} \Big)^{1/2} \leq C \h \Vert \uu \Vert_{\Hm{{5/2}}(\dom)}.
\end{align*}
\ifdetail
\mzdet{
\begin{align}
 \Big( \sumk \hk^4 \Vert \uu \Vert^4_{\Hm{{5/2}}(\telem)} \Big)^{1/4} = \Big( \sumk (\hk^2 \Vert \uu \Vert^2_{\Hm{{5/2}}(\telem)})^2 \Big)^{1/4} \leq \Big(\Big( \sumk \hk^2 \Vert \uu \Vert^2_{\Hm{{5/2}}(\telem)} \Big)^2\Big)^{1/4}
\end{align}
}
\fi
On the other hand using the inverse inequality \eqref{Eq::Inverse} one has
\begin{align*}
\Vert \dser \Vert_{\Wmp{4}{1}(\dom, \triag)} &\leq \Big(\sumk \Cinv^2 \hk^{-1}   \Vert \dser \Vert^2_{\Wmp{2}{1}(\telem)} \Big)^{1/2}.
\end{align*}
\ifdetail
\mzdet{
\begin{align*}
\Vert \dser \Vert_{\Wmp{4}{1}(\dom, \triag)} &= \Big(\sumk \Vert \dser \Vert^4_{\Wmp{4}{1}(\telem)} \Big)^{1/4}  \nonumber \\ 
&\leq \Big(\sumk \big( \Cinv \hk^{-1/2}   \Vert \dser \Vert_{\Wmp{2}{1}(\telem)}\big)^4 \big)^{1/4} =  \Big(\sumk \big( \Cinv^2 \hk^{-1}   \Vert \dser \Vert^2_{\Wmp{2}{1}(\telem)}\big)^2 \big)^{1/4}\nonumber \\ 
& \leq \Big(\big(\sumk  \Cinv^2 \hk^{-1}   \Vert \dser \Vert^2_{\Wmp{2}{1}(\telem)}\big)^2 \big)^{1/4} \nonumber \\  
& = \Big(\sumk \Cinv^2 \hk^{-1}   \Vert \dser \Vert^2_{\Wmp{2}{1}(\telem)} \big)^{1/2}.
\end{align*}
}
\fi
Then, application to \eqref{Eq::Energy_dserr} for term $\Vert \dser \Vert_{\Wmp{2}{1}(\telem)}$  with $\mu = \min \{ \smth, \poldeg+1 \}$ and $\mu_* = \min \{ \smth_*, \poldeg+1 \}$, and quasi-uniformity condition \eqref{Eq::quasi_unif} yield
\begin{align}
\label{Eq::tmp_dser}
\Vert \dser \Vert_{\Wmp{4}{1}(\dom, \triag)} &\leq C  \Big( \sumk \hk^{2(\mu- 1) - 1}  \Vert \uu \Vert^2_{\Hm{\smth}(\telem)}  +\sumk \hk^{2\mu_* - 1}  \Vert \diff( \uu, \nabla \uu) \Vert^2_{\Hm{{\smth_*}}(\telem)}\Big)^{1/2} \nonumber \\
& \leq C \h^{\mu - 3/2} \Vert \uu \Vert_{\Hm{\smth}(\dom, \triag)} + C \h^{\mu_* - 1/2}  \Vert \diff( \uu, \nabla \uu) \Vert_{\Hm{{\smth_*}}(\dom, \triag)}.
\end{align} 
Hence, $\Vert \dser \Vert_{\Wmp{4}{1}(\dom, \triag)} \leq C \h$ for $\uu \in \Hm{{5/2}}(\dom, \triag)$ and $\poldeg \geq 2$, provided that $ \diff( \uu, \nabla \uu) \in \Hm{{3/2}}(\dom, \triag)$. {Similarly, since $ \liftR(\jump{\er}) \in \Sigsp$ one can employ \eqref{Eq::Inverse} to write
\begin{align*}
\Vert \liftR(\jump{\er}) \Vert_{\Lp{4}(\dom)} &\leq \Big(\sumk \Cinv^2 \hk^{-1}   \Vert \liftR(\jump{\er}) \Vert^2_{\Lp{2}(\telem)} \Big)^{1/2}.
\end{align*} 
Using  \eqref{Eq::liftcauchy}, \eqref{Eq::quasi_unif}, Corollary~\ref{Cor::Energy_error}, and with similar arguments as \eqref{Eq::tmp_dser} yield
\begin{align}
 \Vert \liftR(\jump{\er}) \Vert_{\Lp{4}(\dom)} \leq C \h^{\mu - 3/2} \Vert \uu \Vert_{\Hm{\smth}(\dom, \triag)} + C \h^{\mu_* - 1/2}  \Vert \diff( \uu, \nabla \uu) \Vert_{\Hm{{\smth_*}}(\dom, \triag)} \leq C \h.
\end{align}
}
Combining all terms we have
\begin{equation}
\vert \mathcal{N}(\uu, \uh, \psi) \vert \leq C \h \ennorm{\uu - \uh} \Vert \psi \Vert_{\Hm{2}(\dom)} 
\end{equation}

\item {Using Lemma~\ref{Lem::Cnsst}	 and noticing that  $ \prmcnst(\uu, \psi) = 0$ (see \eqref{eq:consErrPrimal} for smooth $\psi$ as well as the boundary condition of \eqref{Eq::Adjoint}), one has the following upper bound for $\prmcnst(\uu,  \psi_h)$
\begin{equation}
\vert \prmcnst(\uu,  \psi_h)  \vert = \vert \prmcnst(\uu,  \psi - \psi_h)  \vert   \leq  C_{con} \Big( \sumk \h_\telem^{2\mu_*}  \Vert \diff( \uu, \nabla \uu) \Vert^2_{\Hm{{\smth_*}}(\telem)} \Big)^{1/2} \ennorm{ \psi - \psi_h}.
\end{equation}
The approximation result of Lemma~\ref{Lem::hp_approx} and $\Hm{2}$-regularity of $\psi$ give
\begin{equation}
\vert \prmcnst(\uu,  \psi_h)  \vert  \leq  C \h \Big( \sumk \h_\telem^{2\mu_*}  \Vert \diff( \uu, \nabla \uu) \Vert^2_{\Hm{{\smth_*}}(\telem)} \Big)^{1/2} \Vert \psi \Vert_{\Hm{2}(\dom)}.
\end{equation}
}

\item 
For the adjoint consistency error $\dualcnst(\uu, \er, \psi)$, from Lemma~\ref{Lemm::Adjoint_Consst} one has
\begin{align*}
\vert \dualcnst(\uu, \er, \psi)  \vert &  \leq \Cacon \h \Vert \diff_z( \uu, \nabla \uu) \Vert_{\Wmp{\infty}{1}(\dom)}  \Vert \psi \Vert_{\Hm{2}(\dom)} \ennorm{\uu - \uh},
\end{align*}
\end{enumerate}
Combining steps (i)-(iii), with the energy error estimate in Corollary~\ref{Cor::Energy_error}, one can write
\begin{align*}
\vert \mathcal{N}(\uu, \uh, \psi) \vert + \vert \prmcnst(\uu,  \psi_h) \vert  +  \vert \dualcnst(\uu, \er, \psi)  \vert \leq C \h \Vert \psi \Vert_{\Hm{2}(\dom)} \Big(\! \! \sumk \hk^{2(\mu - 1)}  \Vert \uu \Vert^2_{\Hm{\smth}(\telem)} + \hk^{2\mu_*}  \Vert \diff( \uu, \nabla \uu) \Vert^2_{\Hm{{\smth_{*}}}(\telem)}  \Big)^{1/2}
\end{align*}

The remaining term in \eqref{Eq::ltwo_err_tmp} can be bounded using Lipschitz continuity \eqref{Eq::Lip} and Lemma~\ref{Lem::Approx}
\begin{equation}
\difffrm(\uh,\psi - \psi_h) - \difffrm(\uu,\psi - \psi_h) \leq \CLC \ennorm{\uh-\uu} \ennorm{\psi-\psi_h} \leq \CLC \CA' \h \ennorm{\uh-\uu} \Vert \psi \Vert_{\Hm{2}(\dom)}.
\end{equation}
Application to the elliptic regularity of $\psi$ and Corollary~\ref{Cor::Energy_error} gives
\begin{align*}
	\ltwonrm{\uh-\uu}{\dom} &  \leq C \h \Big( \sumk \hk^{2(\mu - 1)}  \Vert \uu \Vert^2_{\Hm{\smth}(\telem)} + \hk^{2\mu_*}  \Vert \diff( \uu, \nabla \uu) \Vert^2_{\Hm{{{\smth_{*}}	}}(\telem)}  \Big)^{1/2}
\end{align*} 
which can be summarized as the following lemma for the optimal $\Ltwo$ error estimate
\begin{lemma}
Assume  that $\uh$ and $\uu$ are the solutions of \eqref{Eq::primal} and the exact solution of \eqref{Eq::origPDE}, respectively. Also assume that $\uu \in \Hm{\smth}(\dom, \triag) $ and $\diff( \uu, \nabla \uu) \in \Hm{{\smth_*}}(\dom, \triag)$ with $\smth \geq \frac{5}{2}$ and $\smth_* \geq \frac{3}{2}$. Then, there exists $\Cerr' > 0$ such that the following holds
\begin{equation}
\Vert \uu - \uh \Vert_{\Ltwo(\dom)} \leq \Cerr' \Big( \sumk \hk^{2\mu}  \Vert \uu \Vert^2_{\Hm{\smth}(\telem)}  + \sumk \hk^{2(\mu_*+1)}  \Vert \diff( \uu, \nabla \uu) \Vert^2_{\Hm{{\smth_*}}(\telem)}  \Big)^{1/2},  
\end{equation}
where  $\mu = \min(\smth, \poldeg+1)$ and $\mu_* = \min(\smth_*, \poldeg+1)$ and $\poldeg 
\geq 2$. 
\end{lemma}

\section{Conclusion}
In this work, we have analyzed different DG	 formulations of a quasilinear elliptic problem by introducing appropriate numerical flux functions inspired by their original version in linear problems. We showed that in spite of the fact that all of these formulations are inconsistent, they have the asymptotic consistency property for both primal and dual problem. Moreover, we also proved  the stability of the solution in $\Ltwo$ norm  under mild assumptions on the problem.

Furthermore, for BR2 and SIPG discretizations, we proved the existence and uniqueness of the discrete solution in case of monotone and globally Lipschitz diffusion operator. Afterwards, under regularity assumptions for the exact solution, we proved the optimal convergence rate in energy norm as well as $\Ltwo$ norm.

\bibliographystyle{siam}
\bibliography{mybib}

\end{document}